\documentclass[final,leqno,letterpaper]{article}
\usepackage[utf8]{inputenc}
\usepackage[top=1in, bottom=0.95in, left=0.9in, right=0.9in]{geometry}

\usepackage{xcolor}
\usepackage{graphicx}

\newcommand{\bfS}{\mathbf{S}}

\newcommand{\bfH}{\mathbf{H}}
\newcommand{\bfR}{\mathbf{R}}

\newcommand{\bfD}{\mathbf{D}}
\newcommand{\bfB}{\mathbf{B}}

\newcommand{\bfx}{\mathbf{x}}
\newcommand{\bfy}{\mathbf{y}}

\newcommand{\Rp}{\bfR_{RR}}

\newcommand{\Rm}{\bfR_{ME}}

\newcommand{\bfV}{\mathbf{V}}

\newcommand{\bfLam}{\boldsymbol{\Lambda}}

\usepackage{bm}
\usepackage{mathtools}
\usepackage{amsmath}
\usepackage{amsthm}
\usepackage{algorithmicx}

\usepackage{amsfonts}
\newcommand{\bfI}{\mathbf{I}}

\usepackage{url}

\def\lf{\left\lfloor}   
\def\rf{\right\rfloor}

\DeclareMathAlphabet\mathbfcal{OMS}{cmsy}{b}{n}

\newcommand{\bfM}{\mathbf{M}}

\newcommand{\be}{\mathbf{e}}
\newcommand{\bfL}{\mathbf{L}}
\newcommand{\bfQ}{\mathbf{Q}}

\newcommand{\bfA}{\mathbf{A}}
\newcommand{\bfPhi}{\boldsymbol\Phi}
\newcommand{\bfPsi}{\boldsymbol\Psi}
\newcommand{\bv}{\mathbf{v}}
\newcommand{\LM}{\mathbf{L}_M}
\newcommand{\rev}{\color{black}}

\newtheorem{mydef}{Definition}
\newtheorem{prop}{Proposition}
\newtheorem{theorem}{Theorem}
\newtheorem{lem}{Lemma}

\newtheorem{remark}{Remark}
\usepackage{algorithm2e}
\RestyleAlgo{boxruled}

\title{Saddle point preconditioners for weak-constraint 4D-Var}
\author{Jemima M. Tabeart\footnotemark[1] \and John W. Pearson\footnotemark[2]}

\begin{document}
	\renewcommand{\thefootnote}{\fnsymbol{footnote}}
	\footnotetext[1]{Department of Mathematics and Computer Science, Eindhoven University of Technology, De Zaale, Eindhoven, 5612 AZ, The Netherlands (\texttt{j.m.tabeart@tue.nl})}
	\footnotetext[2]{School of Mathematics, The University of Edinburgh, James Clerk Maxwell Building, The King’s Buildings, Peter Guthrie Tait Road, Edinburgh, EH9 3FD, United Kingdom (\texttt{j.pearson@ed.ac.uk})}
	\maketitle
	\begin{abstract}
		{{Data assimilation algorithms combine information from observations and prior model information to obtain the most likely state of a dynamical system. The linearised weak-constraint four-dimensional variational assimilation problem can be reformulated as a saddle point problem, which admits more scope for preconditioners than the primal form. In this paper we design new terms which can be used within existing preconditioners, such as block diagonal and constraint-type preconditioners. 
				Our novel preconditioning approaches: (i) incorporate model information, and (ii) are designed to target correlated observation error covariance matrices. To our knowledge (i) has not previously been considered for data assimilation problems. We develop new theory demonstrating the effectiveness of the new preconditioners within Krylov subspace methods. Linear and non-linear numerical experiments reveal that our new approach leads to faster convergence than existing state-of-the-art preconditioners for a broader range of problems than indicated by the theory alone. We present a range of numerical experiments performed in serial}}.
	\end{abstract}
	\section{Introduction}
	
	Data assimilation has seen substantial interest in fields such as numerical weather prediction \cite{Carrassi18,rawlins07}, ecology \cite{PinningtonEwanM.2016Itro,PinningtonEwanM.2017Uteo}, and hydrology  \cite{COOPER2018199,waller18} in recent decades. The variational data assimilation problem can be written mathematically as follows:
	For a given time window $[t_0,t_N]$, let $\bfx^t_i\in \mathbb{R}^s$ be the true state of a dynamical system of interest at time $t_i$, where $s$ is the number of state variables. 
	Data assimilation algorithms combine observations of a dynamical system, $\bfy_i\in\mathbb{R}^{p_i}$ at times $t_i$, with prior information from a model, $\bfx_b \in \mathbb{R}^s$, to find $\bfx_i\in \mathbb{R}^s$, the most likely state of the system at time $t_i$.  The prior, or background state, is valid at initial time $t_0$ and can be written as an approximation to the true state via $\bfx_b = \bfx^t_0 + \epsilon^b$. We assume that the background errors $\epsilon^b\sim\mathcal{N} (0,\bfB) $, where $\bfB\in\mathbb{R}^{s\times s}$ is a background error covariance matrix. 
	In order to compare observations made at different locations, or of different variables to those in the state vector $\bfx_i$,  we define a, possibly non-linear, observation operator $\mathcal{H}_i:\mathbb{R}^{s} \rightarrow \mathbb{R}^{p_i}$ which maps from state variable space to observation space at time $t_i$. {\rev Observations at time $t_i$ are written as $\bfy_i = \mathcal{H}_i[\bfx^t_i]+\epsilon_i \in\mathbb{R}^{p_i}$}, for $i=0,1,\dots,N$, where the observation error $\epsilon_i\sim \mathcal{N}(0,\bfR_i)$ and $\bfR_i\in\mathbb{R}^{p_i\times p_i}$ are observation error covariance matrices.

	In weak-constraint four-dimensional varational data assimilation (4D-Var) the state $\bfx_{i-1}$ at time $t_{i-1}$  is propagated to the next observation time $t_{i}$ using an imperfect forecast model, $\mathcal{M}_i$, to obtain
	$\bfx_i = \mathcal{M}_i(\bfx_{i-1})+\epsilon^m_i.$
	The model error at each time is given by $\epsilon^m_i\sim\mathcal{N} (0,\bfQ_i) $, where $\bfQ_i\in\mathbb{R}^{s\times s}$ is the model error covariance matrix at time $t_i$. It is typically assumed that the error covariance matrices are mutually uncorrelated across different types and different observation times.
	The analysis, or most likely state $\bfx_0$ at time $t_0$, minimises the weak constraint 4D-Var objective function
	given by 
	\begin{align}
		\nonumber J(\bfx_0,\bfx_1,\dots,\bfx_N)={}&(\bfx_0-\bfx_b)^\top\mathbf{B}^{-1}(\bfx_0-\bfx^b)+\sum_{i=0}^{N}(\bfy_i-\mathcal{H}_i[\bfx_i])^\top\bfR_i^{-1}(\bfy_i-\mathcal{H}_i[\bfx_i])\\ 
		\label{eq:CostFn} &{}+\sum_{i=1}^N(\bfx_{i}-\mathcal{M}_i(\bfx_{i-1}))^\top\bfQ_{i}^{-1}(\bfx_{i}-\mathcal{M}_i(\bfx_{i-1})). 
	\end{align}

	Weak-constraint 4D-Var is used in numerical weather prediction (NWP) to estimate the initial condition for a weather forecast {\color{black}\cite{tr2006accounting,tremolet2007model}}. Practical implementations of \eqref{eq:CostFn} pose mathematical and computational challenges. Firstly, the dimension of the problem can be large: in NWP \cite{Carrassi18} the dimension of the state can be of order $10^9$, and the number of observations can be of order $10^6$. This application is also time critical, so the time that can be allocated to the data assimilation procedure is very limited in operational situations. Computational efficiency is therefore vital, and designing techniques to ensure fast convergence of the minimisation of the objective function has been an ongoing area of research interest, see for instance \cite{fisher2017parallelization,bonavita2017strategy,ElAkkraoui13,Gurol14}.

	The weak constraint objective function \eqref{eq:CostFn} is typically solved via an incremental approach, where a small number of non-linear outer loops and a larger number of linearised inner loops are solved  \cite{GrattonS.2007AGMf,gratton2011range}. Standard Krylov subspace solvers can then be applied for the inner loops. An alternative approach involves solving the linearised inner loop using a saddle  point formulation \cite{gratton2018guaranteeing,fisher2017parallelization,green2019model,dauvzickaite2020spectral}, {which} admits a richer choice of preconditioning structures compared to the primal form. 
	Prior work has developed specific preconditioners for the saddle point data assimilation problem \cite{gratton2018guaranteeing,fisher2017parallelization}, typically focusing on approximations to the term containing information about the linearised model. Block diagonal preconditioners are appealing, due to the potential to apply the MINRES algorithm and develop guaranteed theoretical insight about the convergence rates based on the eigenvalues of the preconditioned system, although inexact constraint preconditioners with GMRES have been found to yield better performance in data assimilation settings for a variety of problems \cite{gratton2018guaranteeing,freitag2018low}. 
	
	Recent research on the primal form of the variational data assimilation problem has revealed that the observation error covariance matrices $\bfR_i$ play an important role in determining the convergence of iterative methods \cite{tabeart2018conditioning,tabeart2020conditioning}. In the last {\rev{decades}}, researchers have increasingly made use of observing systems that have correlated observation error covariance matrices \cite{weston14,stewart08}.  Previous saddle point preconditioners typically applied the exact inverses of $\bfR_i$ \cite{gratton2018guaranteeing,green2019model}, which is known to be computationally infeasible for many satellite observing systems \cite{weston14,tabeart2020improving}. It is therefore expected that new terms within saddle point preconditioners which incorporate correlated information from $\bfR_i$ and are inexpensive to apply could be beneficial in terms of convergence for many observing systems. {\rev{In what follows, we therefore consider preconditioners for the observation error covariance matrix that explicitly allow for full correlation matrices.}}

	{\rev{In this paper, we consider novel terms within existing preconditioning structures for the saddle point framework, with particular focus on the correlated observation error setting}}. 
	To our knowledge, our preconditioners account explicitly for model information within the model term for the first time. We begin in Section~\ref{sec:Background} by introducing the saddle point formulation of the weak-constraint 4D-Var data assimilation problem, and presenting existing state-of-the-art preconditioners for the saddle point setting.  In Section~\ref{sec:EvalBounds} we prove bounds on the eigenvalues of a block diagonal preconditioned system in terms of constituent matrices of the saddle point problem. In Sections~\ref{sec:ApproxLhat} and \ref{sec:Rprecond} we then analyse the effect of applying existing and new preconditioners for the model term and observation error covariance term respectively. In Section~\ref{sec:NumFramework} we present {\rev{our numerical framework for the experimental results in~Section \ref{sec:NumExpt}}}. 
	For both the  non-linear Lorenz 96 problem and linear heat equation problem we find that our new preconditioners result in a reduced iteration count compared to block diagonal and {inexact constraint} preconditioners in a variety of settings, {{in the presence of correlated observation errors.}} 
	Although the theoretical guarantees only apply to the block diagonal preconditioner, qualitative behaviour is similar for the inexact constraint preconditioner, with large reductions in iterations and {matrix--vector products}. 
	Finally, in Section~\ref{sec:Conclusions} we present our conclusions.

	\section{Background}\label{sec:Background}
	
	\subsection{Data assimilation}

	In this section we introduce the saddle point formulation of the weak constraint four-dimensional variational (4D-Var) data assimilation problem \eqref{eq:CostFn}. Given a time window $[t_0,t_N]$, split into $N$ subwindows, we wish to find the state $\bfx\in \mathbb{R}^s$ at time $t_0$ that is closest in a weighted norm sense both to the observations throughout the time window, and to prior information at the initial time propagated by a model.
	The incremental primal formulation updates $\bfx_0^{(l+1)} = \bfx_0^{(l)}+\boldsymbol\delta\bfx^{(l)}$ by solving an objective function via a series of inner and outer loops to find a sequence of increments to the background state $\bfx_b = \bfx_0^{(0)}$. In the inner loop a linearised problem is solved, typically using iterative solvers such as the Conjugate Gradient method \cite{cg}, and the outer loop is used to update linearisations of model and observation operators.
	
	For each outer loop $l$, the inner loop minimises a quadratic objective function to find  $\boldsymbol\delta\bfx^{(l)}\in\mathbb{R}^{s( N+1)}$, {{where $\boldsymbol\delta\bfx^{(l)} = \bfx^{(l+1)}-\bfx^{(l)}$. 
			Writing $\boldsymbol\delta\bfx = (\boldsymbol\delta\bfx_0^\top,\boldsymbol\delta\bfx_1^\top,\dots,\boldsymbol\delta\bfx_N^\top)^\top$,}} the full non-linear observation operator $\mathcal{H}_i$ (similarly the model operator $\mathcal{M}_i$) is linearised about the current best guess $\bfx_i^{(l)}$ to obtain the linearised operator $\bfH^{(l)}_i$ (respectively $\bfM^{(l)}_i$). 
	The updated initial guess $\boldsymbol\delta\bfx_0^{(l)}$ is propagated forward {\color{black}between observation times} by $\bfM^{(l)}_i$ to obtain $\boldsymbol\delta\bfx^{(l)}_{i+1} = \bfM^{(l)}_i\boldsymbol\delta\bfx^{(l)}_{i}$. \color{black} We note that the time between observations is likely to consist of multiple numerical model time-steps, hence $\bfM_i^{(l)}$ often corresponds to the composition of many discretised models for a single observation time-step. 
	
	\color{black} Alternatively, the quadratic objective function in the inner loop may be replaced with a saddle point system. Following the notation of \cite[Eq. (3.17)]{green2019model}, we substitute the linearised objective function with the following saddle point system:
	\begin{equation}\label{eq:saddlepointsystem}
		\begin{pmatrix}
			\mathbf{D} & \mathbf{0} & \mathbf{L} \\
			\mathbf{0} & \mathbf{R} & \mathbf{H} \\
			\mathbf{L}^{\top} & \mathbf{H}^{\top} & \mathbf{0}\end{pmatrix}
		\begin{pmatrix}\boldsymbol\delta \boldsymbol\eta \\
			\boldsymbol\delta \boldsymbol{\nu} \\
			\boldsymbol\delta \mathbf{x}
		\end{pmatrix}=\begin{pmatrix}
			\mathbf{b} \\
			\mathbf{d} \\
			\mathbf{0}
		\end{pmatrix}.
	\end{equation}

	In this paper we focus on new preconditioners for the saddle point coefficient matrix:
	\begin{equation}\label{eq:saddlepointA}
		\mathbfcal{A}=\begin{pmatrix}
			\mathbf{D} & \mathbf{0} & \mathbf{L} \\
			\mathbf{0} & \mathbf{R} & \mathbf{H} \\
			\mathbf{L}^{\top} & \mathbf{H}^{\top} & \mathbf{0}\end{pmatrix} \in \mathbb{R}^{(2s+p)(N+1)\times(2s+p)(N+1)},
	\end{equation}
	where $\bfD$, $\bfR$, $\bfH$ are the following block diagonal matrices:
	\begin{align*}
		\bfD={}&\texttt{blkdiag}\left(\bfB,\bfQ_1,\bfQ_2,\dots,\bfQ_N\right)\in \mathbb{R}^{(N+1)s\times (N+1)s}, \\
		\bfR={}&\texttt{blkdiag}\left(\bfR_0,\bfR_1,\bfR_2,\dots,\bfR_N\right)\in \mathbb{R}^{(N+1)p\times (N+1)p}, \\
		\bfH={}&\texttt{blkdiag}\left(\bfH^{(l)}_0,\bfH^{(l)}_1,\bfH^{(l)}_2,\dots,\bfH^{(l)}_N\right)\in \mathbb{R}^{(N+1)p\times (N+1)s},
	\end{align*}
	with the $i$th diagonal block in each case being the relevant matrix for time $t_i$. Here, the {\scshape Matlab}-style notation `\texttt{blkdiag}' is used to describe a block diagonal matrix in terms of its block diagonal entries.

	The matrix $\bfL\in \mathbb{R}^{(N+1)s\times (N+1)s}$ contains the linearised model information that evolves $\bfx_0$ {{between the}} $N$ {{observation times, or subwindows}}, written 
	\begin{equation}\label{def:L}
		\mathbf{L}=\begin{pmatrix}
			\mathbf{I} & & & &\\
			-\mathbf{M}^{(l)}_{1} & \mathbf{I} & & &\\
			& -\mathbf{M}^{(l)}_{2} & \mathbf{I} & & \\
			& & \ddots & \ddots & \\
			& &&  -\mathbf{M}^{(l)}_{N} &\mathbf{I}  
		\end{pmatrix},
	\end{equation}
	where $\mathbf{I}$ denotes the $s\times s$ identity matrix. As we consider preconditioners for the inner loop only, for the remainder of the paper we drop the superscripts that denote the outer loop iteration, and simply use $\bfH_i$ and $\bfM_i$.
	
	As the saddle point system is indefinite, methods such as the Conjugate Gradient algorithm cannot be used. Depending on the choice of preconditioner, MINRES \cite{minres} or GMRES \cite{gmres} are examples of viable algorithms for solving linear systems of the form \eqref{eq:saddlepointsystem}. We note that one challenge of the saddle point formulation is that monotonic decrease of the objective cost is no longer guaranteed \cite{gratton2018guaranteeing}. This can be challenging in operational settings where a very limited number of iterations are performed. In this paper we design new preconditioners and study their numerical performance when iterative methods are permitted to reach convergence. { The goal is the design of sufficiently effective and efficient preconditioners to allow convergence of MINRES or GMRES approaches in an operational setting.}

	Numerical methods for saddle point systems are well-studied in the optimisation literature (see \cite{BGL} for a comprehensive survey). In order to devise suitable approximations of such systems, as in the forthcoming section, one powerful approach is to approximate the `leading' $(1,1)$ block of the matrix, along with its Schur complement \cite{Kuznetsov,MGW}. For saddle point problems arising from data assimilation, not only does the $(1,1)$ block often have complex structure, but the constraint block contains evolution of the model terms forward/backward in time, leading to a Schur complement which is very difficult to approximate cheaply \cite{fisher2017parallelization}. Therefore, approximating the constraint block cheaply is an important consideration for a good preconditioner of the matrix \eqref{eq:saddlepointA}. The work presented here therefore attempts to combine suitable approximations for the $(1,1)$ block, the constraint block, and hence the Schur complement.

	\subsection{Preconditioners for saddle point systems from data assimilation}
	We now introduce some preconditioners that have been applied to the saddle point formulation of the data assimilation problem described above. We start by considering two classes of preconditioner: the block diagonal preconditioner and the inexact constraint preconditioner. {\color{black} Although the {\color{black} standard constraint approach} \cite{bergamaschi2007inexact,bergamaschi2011erratum} would include $\widehat{\bfH}$, an approximation to $\bfH$, the use of $\widehat{\bfH} = \mathbf{0}$ is popular in the data assimilation setting where it is commonly called the `inexact' constraint preconditioner \cite{fisher2017parallelization,gratton2018guaranteeing,freitag2018low}.}
	In this paper we use the same forms that are considered in \cite{gratton2018guaranteeing}, which are given by 
	\begin{equation}\label{eq:prec}
		\mathbfcal{P}_D = \begin{pmatrix}
			\widehat{\bfD} &\mathbf{0} &\mathbf{0}\\
			\mathbf{0}& \widehat{\bfR} &\mathbf{0}\\
			\mathbf{0} &\mathbf{0} & \widehat{\bfS} \\
		\end{pmatrix},\qquad
		\mathbfcal{P}_I = \begin{pmatrix}
			{\bfD} & \mathbf{0} & \widehat{\bfL}\\
			\mathbf{0} & \widehat{\bfR} & \mathbf{0}\\
			\widehat{\bfL}^\top & \mathbf{0} & \mathbf{0} \\
		\end{pmatrix},
	\end{equation}
	where $\widehat{\bfD}$ and $\widehat{\bfR}$ are approximations to $\bfD$ and $\bfR$ which are easier to apply than the original matrices, and $\widehat{\bfL}$ is an efficient approximation of $\bfL$. The exact (negative) Schur complement is given by
	\begin{equation*}
		\bfS = \bfL^\top\bfD^{-1}\bfL + \bfH^\top\bfR^{-1}\bfH,
	\end{equation*}
	which we may approximate by $\widehat{\bfS}$. For instance, one may drop the second term and take an approximation $\widehat{\bfL}$ to $\bfL$: that is a Schur complement approximation of the form $\widehat{\bfL}^\top\bfD^{-1}\widehat{\bfL}$.
	
	{\color{black}One attraction of the block diagonal preconditioner is that one may guarantee a fixed convergence rate based on the eigenvalues of the preconditioned system $\mathbfcal{P}_D^{-1}\mathbfcal{A}$. In Section~\ref{sec:EvalBounds} we present bounds on the eigenvalues of the preconditioned system for the block diagonal preconditioner. The inexact constraint preconditioner has been found to yield improved convergence for both toy and operational-scale data assimilation problems, compared to the block diagonal preconditioner for the same choices of $\widehat{\bfL}$ and $\widehat{\bfR}$ \cite{fisher2017parallelization,gratton2018guaranteeing,freitag2018low}. However, it is difficult to develop any theoretical results apart from in the unrealistic setting where $\widehat{\bfL} = \bfL$, $\widehat{\bfH} = \bfH$, $\widehat{\bfR} = \bfR$ \cite{fisher2017parallelization}.}

	In this paper we consider new preconditioners $\widehat{\bfL}$ and $\widehat{\bfR}$ for the model and observation error terms. We now present choices of $\widehat{\bfL}$ and $\widehat{\bfR}$ that have been studied before. 
	Applying the exact choice of $\bfL^{-1}$ in a preconditioner can be prohibitively expensive as it requires serial products of model terms and their adjoints \cite{fisher2017parallelization}. 
	{Evaluations of the model acting on a vector can dominate the computational cost of data assimilation algorithms, so reducing the total number of matrix--vector products of the form $\mathbf{M}_i\mathbf{v}$ or $\mathbf{M}_i^{\top}\mathbf{v}$ is an important consideration when designing new preconditioners.} 
	A common choice of preconditioner for $\bfL$ replaces all of the sub-diagonal terms $\bfM_i$ with $\mathbf{0}$ \cite{gratton2018guaranteeing,green2019model}. We shall denote this choice of preconditioner as $\bfL_0$, and note that $\bfL_0$ is the $(N+1)s\times (N+1)s$ identity matrix, hence so is its inverse. 
	
	Another choice of preconditioner that was considered in \cite{fisher2017parallelization,gratton2018guaranteeing} is given by 
	\begin{equation*}
		\bfL_I = \begin{pmatrix}
			\bfI\\
			-\bfI  & \bfI \\
			& -\bfI & \bfI \\
			&& \ddots & \ddots\\
			&&&-\bfI &\bfI
		\end{pmatrix},
	\end{equation*}
	where the sub-diagonal blocks are replaced with the negative $s\times s$ identity matrix. This was found to perform well experimentally when used in the inexact constraint preconditioner for a two-layer quasi-geostrophic model \cite{fisher2017parallelization}. However, it does not include any model information, and may not be expected to perform well for all problems. In Section \ref{sec:ApproxLhat} we introduce a new preconditioner that incorporates model information.
	
	\color{black}
	Within the outer--inner loop structure of the weak-constraint solver, information from previous outer loops can be used to cheaply update the preconditioners in subsequent loops. This has been found to be computationally beneficial \cite{fisher2018low}. Such approaches could be combined with the method considered in this paper, and are expected to provide value. However, we consider preconditioners for the first outer loop where there is no additional information, and hence developing computationally feasible stand-alone preconditioners is crucial.\color{black}

	In most prior work the exact observation error covariance matrix has been used in preconditioners. This is due to the fact that the cost of applying $\bfR^{-1}$ is assumed not to be prohibitive, due to the inherent block diagonal structure of $\bfR$. However, the rising use of more complex correlation structures (such as inter-channel observation errors, see \cite{weston14}) means that  applying $\bfR^{-1}$ exactly is not always affordable and may be a computational bottleneck. One option is to apply a computationally cheap approximation of $\bfR$ within the preconditioner, such as the diagonal of $\bfR$, however this is typically an inaccurate approximation which can significantly delay convergence.  In this paper we consider new classes of preconditioners for the observation error covariance term {which we expect to improve convergence in the presence of correlated observation error}.
	
	\section{Bounds on eigenvalues for the block diagonal preconditioner}\label{sec:EvalBounds}
	
	For a {symmetric positive definite} block diagonal preconditioner of the form $\mathbfcal{P}_D$, it is possible to describe the convergence of the preconditioned MINRES algorithm by analysing the eigenvalues of the preconditioned system $\mathbfcal{P}_D^{-1}\mathbfcal{A}$ (see \cite[Chapter 4]{ESW2014}, for instance). {In the data assimilation setting $\mathbfcal{P}_D$ is symmetric positive definite, as all proposed approximations $\widehat{\mathbf{D}}$, $\widehat{\mathbf{R}}$ as well as the Schur complement approximation $\widehat{\mathbf{S}}$ are themselves defined to be symmetric positive definite.} 
	{{We note that for the preconditioned GMRES algorithm spectral information is insufficient to describe convergence (see e.g. \cite{greenbaum1996any}). Therefore the theoretical bounds presented in Sections \ref{sec:EvalBounds}--\ref{sec:Rprecond} are informative strictly for the block diagonal preconditioner, with experiments for the inexact constraint preconditioner $\mathbfcal{P}_{I}$ being presented in Section \ref{sec:NumExpt} for numerical comparison.}}
	
	The matrices $\mathbfcal{A}$ and $\mathbfcal{P}_D$ defined in \eqref{eq:saddlepointA} and \eqref{eq:prec} can be written as
	\begin{equation}\label{APdiag}
		\mathbfcal{A}=\left(\begin{array}{cc}
			\bfPhi & \bfPsi^\top \\ \bfPsi & \mathbf{0} \\
		\end{array}\right),\qquad\mathbfcal{P}_D=\left(\begin{array}{cc}
			\widehat{\bfPhi} & \mathbf{0} \\ \mathbf{0} & \widehat{\bfS} \\
		\end{array}\right),
	\end{equation}
	where
	\begin{equation*}
		\bfPhi=\left(\begin{array}{cc}
			\bfD & \mathbf{0} \\ \mathbf{0} & \bfR \\
		\end{array}\right),\qquad \bfPsi=\left(\begin{array}{cc}
			\bfL^\top & \bfH^\top \\
		\end{array}\right),
	\end{equation*}
	and with $\widehat{\bfPhi}$ and $\widehat{\bfS}$ approximations of $\bfPhi$ and the (negative) Schur complement $\bfS=\bfPsi\bfPhi^{-1}\bfPsi^\top$. 
	We now denote
	\begin{equation*}
		\widetilde{\bfS}=\bfL^\top \bfD^{-1}\bfL,\qquad\widehat{\bfS}=\widehat{\bfL}^\top \bfD^{-1}\widehat{\bfL},
	\end{equation*}
	where $\widehat{\bfD}$, $\widehat{\bfR}$, $\widehat{\bfL}$ are approximations of $\bfD$, $\bfR$, $\bfL$. For the forthcoming theory, we suppose that $\bfD$, $\widehat{\bfD}$, $\bfR$, $\widehat{\bfR}$, $\bfS$, $\widetilde{\bfS}$, $\widehat{\bfS}$ are symmetric positive definite, with
	\begin{align*}
		&\lambda(\widehat{\bfD}^{-1}\bfD)\in[\lambda_\bfD,\Lambda_\bfD],\qquad\lambda(\widehat{\bfR}^{-1}\bfR)\in[\lambda_\bfR,\Lambda_\bfR],\\ &\lambda(\widetilde{\bfS}^{-1}\bfS)\in[\lambda_\bfS,\Lambda_\bfS],\qquad\lambda\left((\widehat{\bfL}^\top\widehat{\bfL})^{-1}(\bfL^\top \bfL)\right)\in[\lambda_\bfL,\Lambda_\bfL],
	\end{align*}
	where $\lambda(\cdot)$ denotes the eigenvalues of a matrix. We may then prove the following result:
	
	\begin{theorem}\label{theorem:BDPrec}
		With the definitions as stated above, the eigenvalues of $\mathbfcal{P}_D^{-1}\mathbfcal{A}$ are real, and satisfy:
		\begin{align*}
			\lambda(\mathbfcal{P}_D^{-1}\mathbfcal{A})\in{}&\left[\frac{\lambda_{\bfPhi}-\sqrt{\lambda_{\bfPhi}^2+4\Lambda_{\bfPhi}\Lambda_\bfS\Lambda_\bfL\kappa(\bfD)}}{2},\frac{\Lambda_{\bfPhi}-\sqrt{\Lambda_{\bfPhi}^2+\frac{4\lambda_{\bfPhi}\lambda_\bfS\lambda_\bfL}{\kappa(\bfD)}}}{2}\right] \\
			&\cup[\lambda_{\bfPhi},\Lambda_{\bfPhi}]\cup\left[\frac{\lambda_{\bfPhi}+\sqrt{\lambda_{\bfPhi}^2+\frac{4\lambda_{\bfPhi}\lambda_\bfS\lambda_\bfL}{\kappa(\bfD)}}}{2},\frac{\Lambda_{\bfPhi}+\sqrt{\Lambda_{\bfPhi}^2+4\Lambda_{\bfPhi}\Lambda_\bfS\Lambda_\bfL\kappa(\bfD)}}{2}\right],
		\end{align*}
		where $\lambda_{\bfPhi}=\min\{\lambda_\bfD,\lambda_\bfR\}$, $\Lambda_{\bfPhi}=\max\{\Lambda_\bfD,\Lambda_\bfR\}$, and $\kappa(\cdot)$ denotes the condition number of a matrix.
	\end{theorem}
	\begin{proof}
		Applying well-known results (see \cite[p.\,2906]{ReWa11}, \cite[Theorem 4.2.1]{Rees10}), we have 
		\begin{equation}\label{BDPrecBounds}
			\lambda\in\left[\frac{\delta-\sqrt{\delta^2+4\Delta\Phi}}{2},\frac{\Delta-\sqrt{\Delta^2+4\delta\phi}}{2}\right]\cup[\delta,\Delta]\cup\left[\frac{\delta+\sqrt{\delta^2+4\delta\phi}}{2},\frac{\Delta+\sqrt{\Delta^2+4\Delta\Phi}}{2}\right],
		\end{equation}
		where $\delta$, $\phi$ denote the minimum eigenvalues of $\widehat{\bfPhi}^{-1}\bfPhi$, $\widehat{\bfS}^{-1}\bfS$ for a general block diagonal saddle point preconditioner \eqref{APdiag}, and $\Delta$, $\Phi$ represent the corresponding maximum eigenvalues.
		
		Clearly, for this problem $\delta=\lambda_{\bfPhi}$ and $\Delta=\Lambda_{\bfPhi}$. It remains to analyse the minimum and maximum eigenvalues of the preconditioned Schur complement, which we note is within the range of the Rayleigh quotient (for $\mathbf{v}\neq\mathbf{0}$):
		\begin{equation}\label{Ssplitting}
			\frac{\mathbf{v}^\top \bfS\mathbf{v}}{\mathbf{v}^\top\widehat{\bfS}\mathbf{v}}=\frac{\mathbf{v}^\top \bfS\mathbf{v}}{\mathbf{v}^\top\widetilde{\bfS}\mathbf{v}}\cdot\frac{\mathbf{v}^\top\widetilde{\bfS}\mathbf{v}}{\mathbf{v}^\top\widehat{\bfS}\mathbf{v}}=\frac{\mathbf{v}^\top \bfS\mathbf{v}}{\mathbf{v}^\top\widetilde{\bfS}\mathbf{v}}\cdot\frac{\mathbf{v}^\top \bfL^\top \bfD^{-1}\bfL\mathbf{v}}{\mathbf{v}^\top \bfL^\top \bfL\mathbf{v}}\cdot\frac{\mathbf{v}^\top \bfL^\top \bfL\mathbf{v}}{\mathbf{v}^\top\widehat{\bfL}^\top\widehat{\bfL}\mathbf{v}}\cdot\frac{\mathbf{v}^\top\widehat{\bfL}^\top\widehat{\bfL}\mathbf{v}}{\mathbf{v}^\top\widehat{\bfL}^\top \bfD^{-1}\widehat{\bfL}\mathbf{v}}.
		\end{equation}
		Observing that
		\begin{align*}
			&\frac{\mathbf{v}^\top \bfS\mathbf{v}}{\mathbf{v}^\top\widetilde{\bfS}\mathbf{v}}\in[\lambda_\bfS,\Lambda_\bfS],\qquad\frac{\mathbf{v}^\top \bfL^\top \bfD^{-1}\bfL\mathbf{v}}{\mathbf{v}^\top \bfL^\top \bfL\mathbf{v}}=\frac{\mathbf{y}^\top \mathbf{y}}{\mathbf{y}^\top \bfD\mathbf{y}}\in\left[\frac{1}{\Lambda_\bfD},\frac{1}{\lambda_\bfD}\right], \\
			&\frac{\mathbf{v}^\top \bfL^\top \bfL\mathbf{v}}{\mathbf{v}^\top\widehat{\bfL}^\top\widehat{\bfL}\mathbf{v}}\in[\lambda_\bfL,\Lambda_\bfL],\qquad\frac{\mathbf{v}^\top\widehat{\bfL}^\top\widehat{\bfL}\mathbf{v}}{\mathbf{v}^\top\widehat{\bfL}^\top \bfD^{-1}\widehat{\bfL}\mathbf{v}}=\frac{\mathbf{z}^\top \bfD\mathbf{z}}{\mathbf{z}^\top \mathbf{z}}\in[\lambda_\bfD,\Lambda_\bfD],
		\end{align*}
		where $\mathbf{y}=\bfD^{-1/2}\bfL\mathbf{v}\neq\mathbf{0}$, $\mathbf{z}=\bfD^{-1/2}\widehat{\bfL}\mathbf{v}\neq\mathbf{0}$, we may write that
		\begin{equation*}
			\frac{\mathbf{v}^\top \bfS\mathbf{v}}{\mathbf{v}^\top\widehat{\bfS}\mathbf{v}}\in\left[\frac{\lambda_\bfS\lambda_\bfL\lambda_\bfD}{\Lambda_\bfD},\frac{\Lambda_\bfS\Lambda_\bfL\Lambda_\bfD}{\lambda_\bfD}\right]=\left[\frac{\lambda_\bfS\lambda_\bfL}{\kappa(\bfD)},\Lambda_\bfS\Lambda_\bfL\kappa(\bfD)\right].
		\end{equation*}
		Therefore, it holds that $\phi\geq\frac{\lambda_\bfS\lambda_\bfL}{\kappa(\bfD)}$ and $\Phi\leq\Lambda_\bfS\Lambda_\bfL\kappa(\bfD)$. Substituting the bounds for $\delta$, $\Delta$, $\phi$, $\Phi$ into \eqref{BDPrecBounds} then gives the result.
	\end{proof}
	
	{
		\begin{remark}
			Theorem \ref{theorem:BDPrec} is an extension of known results, for example \cite[p.\,2906]{ReWa11} and \cite[Theorem 4.2.1]{Rees10}, in particular an application of this methodology to saddle point systems arising from weak constraint 4D-Var. We highlight that eigenvalue results of this form are important because they lead to concrete convergence properties of MINRES. As in \cite[Theorem 4.14]{ESW2014}, if $\lambda(\mathbfcal{P}_D^{-1}\mathbfcal{A})\in[-\mu_1,-\mu_2]\cup[\mu_3,\mu_4]$, with $\mu_1,\mu_2,\mu_3,\mu_4>0$ such that $\mu_1-\mu_2=\mu_4-\mu_3$, then after $2\ell$ iterations of MINRES:
			\begin{equation*}
				\|\mathbf{r}^{(2\ell)}\|_{\mathbfcal{P}_D^{-1}} \leq 2 \left( \frac{\sqrt{\mu_1 \mu_4}-\sqrt{\mu_2 \mu_3}}{\sqrt{\mu_1 \mu_4}+\sqrt{\mu_2 \mu_3}}\right)^{\ell} \, \|\mathbf{r}^{(0)}\|_{\mathbfcal{P}_D^{-1}},
			\end{equation*}
			where $\mathbf{r}^{(\cdot)}$ denotes the residual vector at a given iteration. Note that both positive or negative intervals for the eigenvalues of $\mathbfcal{P}_D^{-1}\mathbfcal{A}$ can always be stretched such that $\mu_1-\mu_2=\mu_4-\mu_3$ holds, so this result holds without loss of generality. This allows one to ascertain the convergence behaviour of MINRES when we apply approximations of $\bfL^\top \bfL$, $\bfR$ (as well as $\bfD$), including those discussed in Sections \ref{sec:ApproxLhat} and \ref{sec:Rprecond}.
		\end{remark}
	}
	
	\begin{remark}
		Due to the way the Rayleigh quotient is factored in \eqref{Ssplitting}, Theorem \ref{theorem:BDPrec} gives a potentially very weak bound when $\bfD$ is ill-conditioned. We find the main features which affect the quality of the preconditioner, as predicted by the result, are the approximations of $\bfD$ and $\bfR$, the effect of dropping the second term of the Schur complement, and the quality of the approximation of $\bfL^\top \bfL$ (characterised by the eigenvalues of $\widehat{\bfL}^{-\top}\bfL^\top\bfL\widehat{\bfL}^{-1}$). The latter quantity is the subject of the forthcoming analysis.
	\end{remark}

	\section{{ Approximations} $\widehat{\bfL}$}\label{sec:ApproxLhat}
	Theorem \ref{theorem:BDPrec} suggests that the eigenvalues of the preconditioned system are influenced by the quality of the approximation of $\widehat{\bfL}^\top\widehat{\bfL}$ to $\bfL^\top\bfL$ in the block diagonal preconditioner.
	In this section we consider existing and new choices of $\widehat{\bfL}$ and analyse the eigenvalues and structure of  $\widehat{\bfL}^{-\top}\bfL^\top\bfL\widehat{\bfL}^{-1}$, which is similar to $(\widehat{\bfL}^\top\widehat{\bfL})^{-1}(\bfL^\top \bfL)$. The first choice of $\widehat{\bfL}$, $\bfL_0$, has previously been used for saddle point preconditioners for the data assimilation problem \cite{gratton2018guaranteeing,green2019model}. We also propose $\bfL_M$, a new class of parallelisable preconditioners that depends on a user-defined parameter and incorporates model information.

	\subsection{A new preconditioner, $\LM$, and the eigenvalues of $\LM^{-\top}\bfL^\top\bfL\bfL_M^{-1}$}\label{sec:}
	
	We {{begin by defining}} our proposed preconditioner, $\LM$, which incorporates model information explicitly. 
	For a user chosen parameter $1\le k \le N+1$,  every $k$th block { sub-diagonal} of $\LM$ (i.e. $\bfM_k,\bfM_{2k},\bfM_{3k},\dots$) is set equal to $\mathbf{0}$. The other entries of $\LM$ correspond to those of $\bfL$. Formally we write this as in the definition below. 
	\begin{mydef}\label{def:LM}
		Let $k\in \mathbb{N}$, and let $\bfL_M=\bfL_M(k)\in\mathbb{R}^{(N+1)s\times (N+1)s}$ be a block matrix made up of $s\times s$ blocks. For i,j = $1,\dots, N+1$ we define
		\begin{equation*}
			\emph{the } (i,j)\emph{th block of }\LM = \begin{cases}
				\bfI & \emph{if } i=j,\\
				-\bfM_j & \emph{if } i=j+1 \emph{ and } j \emph{ is not divisible by } k,\\
				\mathbf{0}  & \emph{otherwise}.
			\end{cases}
		\end{equation*}
	\end{mydef}
	
	Similarly we can write the inverse of $\bfL_M$ as follows, using straightforward linear algebra:
	\begin{lem}
		Let $k\in \mathbb{N}$, and let $\bfL_M=\bfL_M(k)\in\mathbb{R}^{(N+1)s\times (N+1)s}$ be a block matrix made up of $s\times s$ blocks. For $i,j = 1,\dots, N+1$ we may evaluate that
		\begin{equation*}
			\emph{the } (i,j)\emph{th block of }\LM^{-1} = \begin{cases}
				\bfI & \emph{if } i=j,\\
				\prod_{m=1}^{i-j}\bfM_{i-m} & \emph{if } 1 \leq i-j \leq (i-1)\,\emph{mod}(k), \\
				\mathbf{0}  & \emph{otherwise}.
			\end{cases}
		\end{equation*}
	\end{lem}
	
	We note that $\bfL_M^{-1}$ is lower triangular, and both the number of non-zero blocks of $\LM^{-1}$ and the number of terms in each of the products of $\bfL_M^{-1}$ are controlled by the parameter $k$. Note that $\bfL$, defined in \eqref{def:L}, satisfies $\bfL=\bfL_M(N+1)$ according to this notation. 
	To further justify the effectiveness of this approximation, we now study the eigenvalues of $\bfL_M^{-\top}\bfL^\top\bfL\bfL_M^{-1}$ theoretically. We begin by stating the structure of  $\bfL_M^{-\top}\bfL^\top\bfL\bfL_M^{-1}$ in terms of the linearised model matrices $\bfM_i$, using straightforward linear algebra.
	
	\begin{lem}
		\label{def:A(M)}
		We can write $\bfL_M^{-\top}\bfL^\top\bfL\bfL_M^{-1} = \bfI + \bfA(\bfM)$ where the block entries of $\bfA(\bfM)\in\mathbb{R}^{s(N+1)\times s(N+1)}$ are defined as follows. For $n = 1, \dots,  \lf{\frac{N}{k}}\rf$,  
		\begin{equation*}
			[\bfA(\bfM)]_{i,j} = \begin{cases}
				(\prod_{t=i}^{nk}\bfM_t^\top)(\prod_{q=j}^{nk}\bfM_{nk-q+j}) & \emph{ for } (n-1)k+1\le i,j\le nk,\\
				-\prod_{t=j}^{nk}\bfM_{nk-t+j} & \emph{ for } i=nk+1, (n-1)k+1\le j \le nk,\\
				-\prod_{t=i}^{nk}\bfM_{t}^\top & \emph{ for } j=nk+1, (n-1)k+1\le i \le nk,\\
				\mathbf{0} & \emph{ otherwise},
			\end{cases}
		\end{equation*}
		{\rev where $[\bfA(\bfM)]_{i,j}$ denotes the $(i,j)$th block of $\bfA(\bfM)$.}
	\end{lem}
	
	We now briefly describe the structure of $\bfA(\bfM)$. The matrix is made up of $\lf{\frac{N}{k}}\rf$ overlapping diagonal blocks, where the size of each block is $(k+1)s\times (k+1)s$. Each block `overlaps' at the $(nk+1,nk+1)$th block of $\bfA(\bfM)$, meaning that the maximum number of non-zero blocks in any row or column is given by $2k+1$. We use this structure to demonstrate that our new preconditioner $\bfL_M$ yields a number of unit eigenvalues for the preconditioned term  $\bfL_M^{-\top}\bfL^\top\bfL\bfL_M^{-1}$.
	\begin{prop}\label{prop:evalsLML}
		Let $\bfL$ be defined as in \eqref{def:L} and $\LM$ as in Lemma \ref{def:A(M)}. For $2\le k\le N+1 $, 
		$\bfL_M^{-\top}\bfL^\top\bfL\bfL_M^{-1}$ has { at least} $rs$ unit eigenvalues where  $r=N+1-2\lf{\frac{N}{k}}\rf$.
	\end{prop}
	\begin{proof}
		From Lemma \ref{def:A(M)} we can construct eigenvectors corresponding to zero eigenvalues of $\bfA(\bfM)$, which will yield unit eigenvalues of  $\bfL_M^{-\top}\bfL^\top\bfL\bfL_M^{-1}$.
		Let $\be_t$ define the canonical vector taking unit value in position $t$ and zero elsewhere. Observe that $\bfA(\bfM)$ is block diagonal matrix with a  $(N-k\lf{\frac{N}{k}}\rf)s\times(N-k\lf{\frac{N}{k}}\rf)s$ zero block in the final position. We can construct $(N-k\lf{\frac{N}{k}}\rf)s$ linearly independent eigenvectors corresponding to the zero eigenvalue for this block using $\be_t$ for $t =(k\lf{\frac{N}{k}}\rf)s+1 ,\dots, Ns$.
		
		For each value of $n$ in Lemma \ref{def:A(M)} we obtain $(k-2)s$ eigenvectors corresponding to a zero eigenvalue, of the form
		\begin{equation*}
			(\mathbf{0}, \dots, \mathbf{0}, \bv_t^{\top}, -(\bfM_r\bv_t)^\top,\mathbf{0},\dots, \mathbf{0})^\top
		\end{equation*}
		for $rs+1\le t \le (r+1)s $ and $nk+2\le r \le (n+1)k-1$.  Therefore these contribute $(k-2)s\lf{\frac{N}{k}}\rf$ linearly independent eigenvectors corresponding to the zero eigenvalue across the whole matrix.
		From the first block in the matrix we obtain an additional $s$ linearly independent eigenvectors corresponding to the zero eigenvalue  via  $(\bv_t^\top,-(\bfM_1\bv_t)^\top,\mathbf{0},\dots,\mathbf{0})^\top$ for $t =1,\dots,s$.
		
		Combining the above reasoning, we obtain $rs$ unit eigenvalues of  $\bfL_M^{-\top}\bfL^\top\bfL\bfL_M^{-1}$ as required, where 
		\begin{equation*}
			r = 1 +  (k-2) \lf{\frac{N}{k}}\rf  + N - k\lf{\frac{N}{k}}\rf= N + 1 - 2\lf{\frac{N}{k}}\rf.
	\end{equation*}  \end{proof}
	
	We see that $r$ { does not decrease} as $k$ increases, i.e. incorporating information from more timesteps will { generally} lead to a larger number of unit eigenvalues of the preconditioned model term. Increasing the number of observation times $N$ will broadly lead to an increase in the number of unit eigenvalues of  $\bfL_M^{-\top}\bfL^\top\bfL\bfL_M^{-1}$, but this behaviour is non-monotonic. 
	
	If we introduce assumptions on the  spectral radii of the model operator terms, we can obtain explicit bounds on the eigenvalues of $\bfL_M^{-\top}\bfL^\top\bfL\bfL_M^{-1}$.
	
	\begin{prop}\label{prop:lowrankallM}
		If $\|\bfM_i\bfM_i^\top\|_2\le 1 \; \forall i$ then the eigenvalues of  $\bfL_M^{-\top}\bfL^\top\bfL\bfL_M^{-1}$ can be bounded above by $k + 1+2\sqrt{k}$.
	\end{prop}
	\begin{proof}
		We bound the eigenvalues of $\bfA(\bfM)$ by splitting the matrix into three sub-matrices $\bfA(\bfM)= \bfA_1+\bfA_2+\bfA_3$, where $\bfA_1,\bfA_2,\bfA_3$ are symmetric and will be defined explicitly in what follows.
		As all matrices being considered are symmetric, using \cite[Fact 5.12.2]{Bernstein} we can bound the maximum eigenvalue of $\bfA(\bfM)$ above by $\lambda_{\max}(\bfA(\bfM))\le \lambda_{\max}(\bfA_1)+\lambda_{\max}(\bfA_2)+\lambda_{\max}(\bfA_3).$
		
		Let $\bfA_1$ be a block diagonal matrix  with blocks of size $nk\times nk$, with entries defined by:
		\begin{equation*}
			\text{the } (i,j)\text{th block of } \bfA_1 = \begin{cases}
				\text{the } (i,j)\text{th block of } \bfA(\bfM)  & \text{for } (n-1)k+1\le i,j\le nk,\\
				\mathbf{0} & \text{otherwise}.
			\end{cases}
		\end{equation*}
		Each $ks\times ks$ block has rank $s${, as the first $(k-1)s$ rows are multiples of the final $s\times s$ rows. Substitution yields the eigenvalue problem 
			$$ \left(\bfM_{nk}^\top\sum_{t=(n-1)k+1}^{nk} \left(\prod_{p=t}^{nk} \bfM_{nk-p+t}\right)\left(\prod_{q=t}^{nk-1} \bfM_q^\top \right)\right) \bv = \mu \bv.$$
			We apply \cite[Theorem 1.3.20]{horn_johnson_1985}, which states that exchanging the order of matrix multiplication for two compatible matrices has no effect on the non-zero eigenvalues of the product, so we instead consider}
		\begin{equation*}
			\left(\sum_{t=(n-1)k+1}^{nk} \left(\prod_{p=t}^{nk} \bfM_{nk-p+t}\right)\left(\prod_{q=t}^{nk} \bfM_q^\top \right)\right) { \bar{\bv}} = \mu { \bar{\bv}}.
		\end{equation*}
		We can separate the contribution of each individual term by applying  \cite[Fact 5.12.2]{Bernstein} 
		to obtain 
		\begin{align*}
			\mu &\le \lambda_{\max}  \left(\sum_{t=(n-1)k+1}^{nk} \left(\prod_{p=t}^{nk} \bfM_{nk-p+t}\right)\left(\prod_{q=t}^{nk} \bfM_q^\top \right)\right) \\
			& \le \sum_{t=(n-1)k+1}^{nk}  \prod_{p=t}^{nk}\|\bfM_p\bfM_p^\top\|_2
			\le\sum_{t=(n-1)k+1}^{nk} 1 = k.
		\end{align*}
		The maximum eigenvalue of each block is bounded above by $k$, and hence $\lambda_{\max}(\bfA_1) \le k$.
		
		Without loss of generality, assume that $\lf{\frac{N}{k}}\rf$ is odd. 
		For $n = 1,3,5,\dots , 2\lf{\frac{N}{2k}}\rf+1$ we define $\bfA_2$ by a block diagonal matrix, with blocks of size $(k+1)s\times(k+1)s$ and entries given by  
		\begin{equation*}
			[\bfA_2]_{i,j}
			= \begin{cases}
				\text{the } (i,j)\text{th block of } \bfA(\bfM) & \text{for } i=nk+1, (n-1)k+1\le j \le nk,\\
				\text{the } (i,j)\text{th block of } \bfA(\bfM)  
				& \text{for } j=nk+1, (n-1)k+1\le i \le nk,\\
				\mathbf{0} & \text{otherwise.}
			\end{cases}
		\end{equation*}

		The $(k+1)s\times(k+1)s$ blocks have rank $2s$ with non-zero eigenvalues that solve
		\begin{equation*}
			\left(\sum_{t=(n-1)k+1}^{nk} \left(\prod_{p=t}^{nk} \bfM_{nk-p+t}\right)\left(\prod_{q=t}^{nk} \bfM_q^\top \right)\right) \bv = \mu^2 \bv.
		\end{equation*}
		
		By the same argument as above we can bound $\mu^2 \le k$. 
		Hence $\lambda_{\max}(\bfA_2) \le \sqrt{k}$.
		
		For $n = 2,4,6,\dots,  2\lf{\frac{N}{2k}}\rf$, $\bfA_3$ is a  block diagonal matrix with blocks of size $(k+1)s\times (k+1)s$ and entries given by 
		\begin{equation*}
			[\bfA_3]_{i,j}= \begin{cases}

				\text{the } (i,j)\text{th block of } \bfA(\bfM) & \text{for } i=nk+1, (n-1)k+1\le j \le nk-1,\\
				
				\text{the } (i,j)\text{th block of } \bfA(\bfM) & \text{for } j=nk+1, (n-1)k+1\le i \le nk-1,\\
				\mathbf{0} & \text{otherwise.}
			\end{cases}
		\end{equation*}
		
		All blocks have the same structure as the blocks of $\bfA_2$, and hence have eigenvalues bounded above by $\sqrt{k}$.

		The largest eigenvalue of $\bfA(\bfM) $ is therefore bounded above $\lambda_{\max}(\bfA(\bfM))\le k + 2\sqrt{k}$. By adding the identity matrix, we obtain the upper bound on the eigenvalues in the proposition statement.
	\end{proof}

	\begin{remark}
		For smaller values of { $N$}, $\bfA_3$ does not enter the working. We can therefore  apply a similar argument with $\bfA = \bfA_1+\bfA_2$ to obtain the tighter bounds { $\lambda(\LM^{-\top}\bfL^\top\bfL\LM^{-1}) \le  1+k+\sqrt{k} ~\emph{ if } k \le { N} < 2k,$ that is $\lf{\frac{N}{k}}\rf = 1$.}
	\end{remark}
	\begin{remark}
		A similar approach is not illustrative when examining a lower bound for the eigenvalues, as this would yield negative bounds, whereas all eigenvalues of $\LM^{-\top}\bfL^\top\bfL\LM^{-1}$ are clearly greater than zero by construction.
	\end{remark}

	\section{{ Approximations} $\widehat{\bfR}_i$}\label{sec:Rprecond}
	Theorem \ref{theorem:BDPrec} suggests that the eigenvalues of the preconditioned system are also influenced by the quality of the approximation of $\widehat{\bfR}$ to $\bfR$ in the block diagonal preconditioner. In this section we consider four choices of $\widehat{\bfR}_i$, {{each of which we apply blockwise to $\bfR_i$}}. Similarly to the preconditioners for $\bfL$, we consider an existing choice of preconditioner that is diagonal, and three new choices of preconditioner that include correlation information.
	We expect the new preconditioners to be beneficial for highly correlated observation error covariance matrices.
	{{Correlated observation errors are currently implemented operationally at a number of numerical weather prediction centres for hyperspectral satellite instruments (see e.g. \cite{weston14,tandeo2020review}) and Doppler Radar Winds (DRW) (e.g. \cite{simonin2019pragmatic}). Hyperspectral instruments have a block covariance structure with a number of highly correlated off-diagonal entries, and DRW error statistics are spatially correlated.}} 
	We focus on the observation error covariance matrix in this section as efficient approximations for this term  have previously been overlooked. In the following we describe the approaches in terms of $\bfR$ without loss of generality, as the methods can also be applied to approximate the blocks of $\bfD$. 
	
	\begin{remark}
		{\rev{We note that $\bfR$ has a block diagonal structure. Each of the preconditioners in this section is applied blockwise, yielding a block diagonal $\widehat{\bfR}$ with at least $N+1$ blocks. The two preconditioners presented in Section~\ref{sec:parallelR} \emph{further} increase the {sparsity} of $\widehat{\bfR}$.}}
	\end{remark}
	
	\subsection{Choices of $\widehat{\bfR}_i$ which increase {sparsity}}\label{sec:parallelR}
	Many previous studies of the saddle point data assimilation formulation assume that $\bfR$ is diagonal or easy to invert \cite{gratton2018guaranteeing,fisher2018low}. {{For those instruments with uncorrelated or diagonally dominant observation error covariance matrices, it is likely that a diagonal approximation $\widehat{\bfR}$ will be sufficient.}} However for more complicated structures, it is unlikely that the exact inverse of $\bfR$ can be applied efficiently in terms of storage or computation. {{The approaches we present here are designed to perform well in terms of effectiveness and efficiency for correlation structures that are currently used operationally.}}
	
	The first two choices of $\widehat{\bfR}_i$ considered in this section  allow for additional {sparsification}  {\rev{of the observation error component of the preconditioners}}. The first preconditioner, denoted $\bfR_{diag}$, takes the diagonal of the original observation error covariance matrix $\bfR_i$. This is often a first approximation of $\bfR_i$ or its inverse for simple covariance structures. This choice of $\widehat{\bfR}_i$ is cheap to apply and {extremely sparse}, but is expected to perform poorly if there is significant off-diagonal correlation structure.

	The second {sparsifying}  choice of $\widehat{\bfR}_i$, denoted $\bfR_{block}$ in Algorithm \ref{alg:Block} { in the Appendix}, is designed to exploit existing block structure in $\bfR_i$. In applications, $\bfR_i$ itself often has a block structure, with the strength of off-diagonal correlations varying, e.g. for different instruments or measurement types {{(see e.g. \cite{weston14})}}. The idea of $\bfR_{block}$ is to retain the sub/super-diagonal blocks of $\bfR_i$ with the largest norm. Neglecting off-diagonal blocks of $\bfR_i$ with smaller norm ensures that $\bfR_{block}$ is decoupled into a block diagonal matrix. 
	Let $\bfR_i\in\mathbb{R}^{p_i\times p_i}$ be a covariance matrix with an associated vector $\texttt{pvec}\in\mathbb{R}^{p_n}$ that specifies the size of `blocks' such that  $\sum_{k=1}^{p_n} \texttt{pvec}_k= p_i$.
	Algorithm \ref{alg:Block} returns a block covariance matrix $\bfR_{block}$ where only off-diagonal blocks with scaled Frobenius norm larger than a user-defined tolerance $\texttt{tol}$ are retained. In order to maximise computational efficiency only norms of the blocks on the first super-diagonal are computed. If two (or more) adjacent blocks are retained, information from higher level super-diagonals is also included. This does not change the {overall} block structure of the new preconditioner, but allows for the inclusion of more information from $\bfR_i$. This will not lead to a large increase in the cost of applying its inverse and we deem that retaining the additional correlation information is likely to be beneficial for the preconditioner.

	\subsection{Preconditioning methods motivated by reconditioning}
	
	The next two choices of $\widehat{\bfR}_i$ are motivated by reconditioning methods \cite{tabeart2020improving}. These are commonly used in data assimilation implementations to mitigate the issues associated with ill-conditioned sample covariance matrices \cite{weston14}. {{These methods do not increase the sparsity of $\widehat{\bfR}_i$ compared to $\bfR_i$, but can be applied to non-block matrices, such as the spatially varying error covariances used for Doppler Radar Winds.}} In this application we consider the use of such methods to develop new terms in the preconditioner only.
	
	\begin{algorithm}[ht]
		\caption{Ridge regression method\label{alg:RR}}
		
		{Inputs: Matrix $\bfR_i$, target condition number $\kappa_{\max}$.}\\
		{Define $\gamma = \frac{\lambda_{\max}(\bfR_i) - \lambda_{\min}(\bfR_i)\kappa_{\max}}{\kappa_{\max}-1} $. \\}
		{Set $\Rp = \bfR_i+\gamma\mathbf{I}$.}
		
	\end{algorithm}
	
	\begin{algorithm}[ht]
		\caption{Minimum eigenvalue method\label{alg:ME}}
		{Inputs: Matrix $\bfR_i = \bfV\bfLam\bfV^\top$, target condition number $\kappa_{\max}$.\\
			Set $\lambda_{\max}(\Rm) = \lambda_{\max}(\bfR_i)$.\\}
		{Define $T=\frac{\lambda_{\max}(\bfR_i)}{\kappa_{\max}}>\lambda_{\min}(\bfR_i)$.\\}
		{ Set the remaining eigenvalues of $\Rm$ via	\begin{equation*}
				\lambda_{k}(\Rm)= 
				\begin{cases}
					\lambda_k(\bfR_i) & \text{if } \lambda_k(\bfR_i)>T,\\
					T & \text{if } \lambda_k(\bfR_i) \le T. \\
				\end{cases}
			\end{equation*}\\}
		{Construct the reconditioned matrix via $\Rm = \textbf{V}\bfLam_{ME}\textbf{V}^\top$, where $\bfLam_{ME}$ is a diagonal matrix with diagonal entries given by $\lambda_k(\Rm)$.}\\
	\end{algorithm}
	
	Algorithms \ref{alg:RR} and \ref{alg:ME} define parameter-dependent preconditioners. Typically in the reconditioning setting $\gamma$ and $T$ are selected such that reconditioned matrices have condition number $\kappa_{\max}$.  However, in the preconditioning approach we select $\gamma$ and $T$ directly, with larger parameter values yielding smaller condition numbers of $\widehat{\bfR}_i$. 
	
	Algorithms \ref{alg:RR} and \ref{alg:ME} can be used to construct preconditioners for $\bfR_i$ that retain much of the structure of the original matrix, but are better conditioned. Additionally, we can prove how the eigenvalues of the preconditioned term that appears in Theorem \ref{theorem:BDPrec} change as we vary $\gamma$ and $T$, respectively. The following { results} determine the spectra of the preconditioned terms $\bfR_{RR}^{-1}\bfR_i$ and $\bfR_{ME}^{-1}\bfR_i$  for any choice of parameters $\gamma$ and $T$. 
	
	\begin{prop}
		Let  $\lambda_k(\bfR_i)$ denote the eigenvalues of $\bfR_i$. The eigenvalues of $\bfR_{RR}^{-1}\bfR_i$ are given by
		\begin{equation*}
			\lambda_k(\bfR_{RR}^{-1}\bfR_i) = \frac{\lambda_k(\bfR_i)}{\lambda_k(\bfR_i)+\gamma}.
		\end{equation*}
	\end{prop}
	\begin{proof}
		Let $\bfR_i = \bfV\bfLam\bfV^\top$, where $\bfLam = \texttt{diag}(\lambda_k)$, be the eigendecomposition of $\bfR_i$. Then
		\begin{equation*}
			\bfR_{RR}  =      \bfV(\bfLam+\gamma\mathbf{I})\bfV^\top
		\end{equation*}
		and hence
		\begin{equation*}
			\bfR_{RR}^{-1}\bfR_i = \bfV\left(\texttt{diag}\left(\frac{\lambda_k}{\lambda_k+\gamma}\right)\right)\bfV^\top.
		\end{equation*}
	\end{proof}
	
	We note that for any value of $\gamma$, $0<\lambda(\bfR_{RR}^{-1}\bfR_i) <1$. For small values of $\gamma$ all eigenvalues are closer to $1$ and as $\gamma$ increases, more eigenvalues move towards zero.

	\begin{prop}
		The eigenvalues of $\bfR_{ME}^{-1}\bfR_i$ are given by \begin{equation*}
			\lambda_k(\bfR_{ME}^{-1}\bfR_i) = \begin{cases}
				1 & \emph{if } \lambda_k>T,\\
				\frac{\lambda_k}{T}<1 & \emph{if } \lambda_k \le T.
			\end{cases}
		\end{equation*}
	\end{prop}
	\begin{proof}
		Let $\bfR_i = \bfV\bfLam\bfV^\top$, where $\bfLam = \texttt{diag}(\lambda_k)$, be the eigendecomposition of $\bfR_i$. Then
		\begin{equation*}
			\bfR_{ME}  =      \bfV\bfLam_{ME}\bfV^\top,
		\end{equation*}
		where $\bfLam_{ME}$ is a diagonal matrix with diagonal entries given by $\max\{T,\lambda_k(\bfR_i)\}$.
		Hence 
		\begin{align*}
			\bfR_{ME}^{-1}\bfR_i  &=      \bfV\left(\texttt{diag}\left(\frac{\lambda_k}{\max\{T,\lambda_k\}}\right)\right)\bfV^\top,
		\end{align*}
		which has eigenvalues corresponding to the expression in the theorem statement.
	\end{proof}
	
	All of the eigenvalues of $\bfR_{ME}^{-1}\bfR_i$ lie in the range $(0,1]$. For small values of $T$ most of the eigenvalues are units, with a few smaller than $1$. As $T$ increases a larger number of eigenvalues become strictly smaller than $1$.

	Therefore both choices of preconditioner maintain the ordering of eigenvalues and yield eigenvalues of the preconditioned matrix lying between $0$ and $1$. For both approaches increasing the parameter leads to larger differences between the {eigenvalues of the preconditioned matrix} and 1. Smaller parameter values may yield choices of $\widehat{\bfR}$ that are {{themselves}} ill-conditioned, and hence expensive to evaluate as part of a preconditioner. {Therefore, there is a balance to be struck when choosing a parameter value to avoid poor conditioning of either the preconditioner $\widehat{\mathbf{R}}$ or the preconditioned term $\widehat{\mathbf{R}}^{-1}{\mathbf{R}}$.}
	A natural question is therefore how to select appropriate parameter values, and how to implement this automatically. Heuristics for automated parameter selection are discussed in Section \ref{sec:NumFramework} for our numerical experiments, but it is likely that some initial investigation would be necessary to identify suitable methods for specific problems of interest. 
	In Section \ref{sec:NumExpt} we compare the performance of the four approximations $\bfR_{diag}$, $\bfR_{block}$, $\bfR_{RR}$, and $\bfR_{ME}$, {{applied block-wise to each block $\bfR_i$ of $\bfR$}}.

	\section{Numerical framework}\label{sec:NumFramework}
	
	In this section we introduce the numerical framework for the experiments presented in Section \ref{sec:NumExpt}.  
	We begin in Section \ref{sec:DAparams} by defining the parameters for the data assimilation problem. The same data assimilation framework is used for all experiments. 
	{\rev{In Section \ref{sec:NLAimplementation} we discuss implementation aspects relating to the preconditioners. }}
	We note that all results are computed using {\scshape Matlab} version 2019b on a machine an a 1.8GHz Intel  Intel quad-core i7 processor with 15GB RAM on an Ubuntu 20.04.2 LTS operating system.

	\subsection{Data assimilation parameters}\label{sec:DAparams}
	We now describe the data assimilation problem that is studied in Section \ref{sec:NumExpt}.
	The size of the state space $s$ is determined by the spatial discretisation $s=\frac{1}{\Delta x}$. 
	For each choice of $s$ we fix the observation operator $\bfH_i \in \mathbb{R}^{p\times s}$ to be the same for all observation times, $i$.  The entries of $\bfH_i$ are all 0 or 1, and each row has a single unit entry. The columns containing units are randomly selected for each choice of $\Delta x$ and ordered. {\rev{We choose}} $p=\frac{s}{2}$, and observations of alternate state variables are smoothed equally over 5 adjacent state variables, with entries either being $0$ or $\frac{1}{5}$. The full observation operator $\bfH$ is then assembled by taking $\bfH = \bfI_{N+1}\otimes \bfH_i$, where $\otimes$ denotes the Kronecker product.
	
	We assume that the model error $\bfQ_i$ is the same at each observation time, i.e. $\bfQ_i \equiv \bfQ$ for $i=1,\dots,N$.
	{\rev{Although this is a simplified choice of model error, we note that treating $\bfQ_i$ is not the focus of this work. More complicated formulations could be taken into account for operational problems, and the preconditioning approaches for $\bf{R}$ discussed in Section \ref{sec:Rprecond} can also be applied to the blocks of $\bfQ$. }} Both $\bfB$ and $\bfQ$ 
	are created using the same routine, based on a SOAR correlation matrix \cite{Yaglom87}. This routine constructs spatial local correlations whilst ensuring that the matrix has high sparsity. 
	Both $\bfB$ and $\bfQ$ are $s\times s$ circulant matrices fully defined by a single row. The number of non-zero entries in each row is fixed irrespective of the value of $s$. The non-zero entries are computed using a modified SOAR function following the procedure in \cite{haben11c}:
	\begin{equation}\label{eq:SOAR}
		c_i = \sigma\left(1+\frac{2|\sin\left(\frac{i\theta}{2}\right)|}{L}\exp\left(-\frac{2|\sin(\frac{i\theta}{2})|}{L}\right)\right), \qquad \theta = \frac{\pi}{\texttt{maxval}},
	\end{equation}
	where $L$ is the correlation lengthscale (0.6 for $\bfB$, 0.5 for $\bfQ$), $\texttt{maxval}$  determines the number of non-zero entries ($100$ for $\bfB$, $120$ for $\bfQ$), and $\sigma$ is the amplitude of the correlation function (0.4 for $\bfB$, 0.2 for $\bfQ$). To ensure positive definiteness of the function, if the smallest eigenvalue of the full matrix is negative a constant $\delta = |\lambda_{\min}(\bfB)|+\psi$ is added to the diagonal, where $\psi$ is a random number in $[0,0.5]$. {{We then assemble $\bfD$ by taking the first block diagonal entry to be $\bfB$ and the remaining block diagonal entries to be $N$ copies of $\bfQ$.}}
	
	The spatial structure for $\bfB$ and $\bfQ$ introduced above is similar to numerical frameworks that have been considered previously for weak-constraint data assimilation experiments \cite{dauvzickaite2020spectral,el2015conditioning}. However, we note that for realistic data assimilation systems the true structure of $\bfQ_i$ is not well known. Improved understanding of model error covariance matrices, and their estimation in preconditioners, is of research interest, but will not be considered here. {{For our experiments we apply the ridge regression preconditioner to $\bfB$ and $\bfQ$ given by Algorithm 2, for the inflation parameter $\delta =0.01$ to obtain $\bfB_{RR}$ and $\bfQ_{RR}$.  We apply $\widehat{\bfD}^{-1}$ using the incomplete Cholesky factors of $\bfB_{RR}$ and $\bfQ_{RR}$, respectively, computed with the {\scshape Matlab} function $\texttt{ichol}$ with zero-fill, i.e. using the same sparsity structure as $\bfB$ and $\bfQ$.}}

	Many experiments that account for correlated observation error covariance matrices use spatial correlations and circulant matrix structures (similar to those we are using for $\bfB,\bfQ$). However, in NWP error correlations often arise from hyperspectral satellite-based instruments which have interchannel uncertainty structures (see for instance \cite{weston14,stewart08}), which appear as block structures within a matrix. Therefore for these experiments we construct a matrix with block structure, designed to replicate many of the properties of realistic interchannel error correlations. 
	
	To construct $\bfR_i\in \mathbb{R}^{p\times p}$ we define two vectors: $\texttt{pvec} \in \mathbb{R}^{\texttt{plen}}$ such that $p = \sum_{k=1}^{\texttt{plen}} \texttt{pvec}_k$ which gives the size of the blocks, and $\texttt{pcorr}\in\mathbb{R}^{\texttt{plen}(\texttt{plen}-1)/2}$ which gives the a multiplication factor for each off-diagonal block (note if $\texttt{pcorr}_k = 0$ the corresponding off-diagonal block is uncorrelated). Diagonal blocks are correlated and constructed as the Hadamard product of a sparse random matrix and a sparse SOAR matrix (using the same approach as for $\bfB$ and $\bfQ$ above). Off-diagonal blocks are sparse random matrices with entries in $(0,\texttt{pcorr}_k)$. The matrix $\bfR_i$ is assembled by adding the diagonal blocks and upper half of the matrix and then symmetrising. 
	This also increases the weight of the diagonal blocks. A similar approach to that used for $\bfB$ and $\bfQ$ is applied to guarantee that $\bfR_i$ is positive definite, where if the minimum eigenvalue of $\bfR_i $ is less than $1$ it is increased to a small positive value. This value is fixed at $0.41$, in order to control the conditioning of $\bfR_i$, and ensure that the condition numbers of $\bfD$ and $\bfR$ are comparable.  {{We  assemble $\bfR$ by taking $\bfR = \bfI_{N+1}\otimes \bfR_i$, where $\otimes$ denotes the Kronecker product and the choice of $\bfR_i\in \mathbf{R}^{p\times p}$ is fixed for a given data assimilation problem.}} This method of construction ensures we have adequate sparsity for high-dimensional experiments, that $\bfR_i$ is positive definite with significant correlations, and that it is well-conditioned. In practice $\bfR$ may be very ill-conditioned compared to $\bfD$, in which case we expect selecting a good choice of $\widehat{\bfR}$ would be even more vital to ensure fast convergence. {{The combination of model and data assimilation parameters used in our experiments is summarised in Table~\ref{tab:ExperimentalDesign}.}}
	
	\begin{table}[]
		\begin{small}
			\centering
			\begin{tabular}{c|ccccc}
				Experiment  & Model & $\bfD$ & $\bfR_i$ & $\bfH$& $p/s$ \\
				\hline
				A & Lorenz 96 & Modified SOAR & Block & Direct observations &0.5\\
				B& Heat equation & Modified SOAR  &Block & Direct observations & 0.25 \\
			\end{tabular}
			\caption{Summary of experimental design for Section \ref{sec:NumExpt}. The modified SOAR function is described in \eqref{eq:SOAR}, the block method to construct $\bfR_i$ is described in Section \ref{sec:DAparams}, as is the direct observation structure of $\bfH$. The final column shows the ratio of observations to state variables at each observation time.}
			\label{tab:ExperimentalDesign}
		\end{small}
	\end{table}

	Finally, we wish to apply the reconditioning-inspired preconditioner to $\bfR_i$ in an online way. 
	One way to do this is to use the small eigenvalues of $\bfR_i$ to select $\gamma $ and $T$. For the ridge regression approach, we set $\gamma = \lambda_{\min}(\bfR_i)$. For the minimum eigenvalue method we compute the smallest two eigenvalue--eigenvector pairs, and set the threshold equal to the second smallest, meaning only a single eigenvalue is changed. The small eigenvalues are computed using \texttt{eigs(R,1,'sr')} in {\scshape Matlab}. For both approaches we use information from $\bfR_i$, but computing a small number of eigenvalues ensures this is computationally efficient. Finally, for all correlated choices of $\widehat{\bfR}_i$ we apply $\widehat{\bfR}_i^{-1}$ using the incomplete Cholesky factors computed with the {\scshape Matlab} function $\texttt{ichol}$ with the same sparsity structure as $\bfR_i$. For $\bfR_{ME}$ this means the reconditioning method is applied as a low-rank update to the Cholesky factors via the Woodbury identity as this is more efficient in terms of storage.

	\begin{table}[]
		\begin{small}
			\centering
			\begin{tabular}{c|cc|cc|cc|cc|c}
				&\multicolumn{2}{c|}{$\bfR_i$}&\multicolumn{2}{c|}{$\bfR_{diag}^{-1}\bfR_i$}&\multicolumn{2}{c|}{$\bfR_{block}^{-1}\bfR_i$}&\multicolumn{2}{c|}{$\bfR_{RR}^{-1}\bfR_i$}&$\bfR_{ME}^{-1}\bfR_i$\\
				$p$ &$\lambda_{\min}$ &$\lambda_{\max}$  &$\lambda_{\min}$ &$\lambda_{\max}$& $\lambda_{\min}$ &$\lambda_{\max}$ & $\lambda_{\min}$ & $\lambda_{\max}$ & $\lambda_{\min}$\\ \hline
				125 &0.4100& 108.63 & 0.0131& 2.7466 & 0.6779 & 1.3221& 0.2908 & 1.0000 & 0.2602 \\
				250 & 0.4100 & 200.55 & 0.0077 & 3.3466 & 0.7886 & 1.2114 & 0.2908 & 1.0000 & 0.6021\\
				500 &0.4100& 312.05 & 0.0061 & 4.4236& 0.8889 & 1.1110 & 0.2908 & 1.0000 & 0.1129 \\
				1250 &0.4100& 540.99 & 0.0044 &5.5440& 0.9335 & 1.0665 & 0.2908 & 1.0000 & 0.1077 \\
				2500 &0.4100& 650.67 & 0.0055 &6.0254& 0.7556 & 1.2444 & 0.2908& 1.0000 & 0.1302\\
				5000 &0.4100& 706.94 & 0.0057 &6.8106 & 0.6058 & 1.3941 & 0.2908 & 1.0000& 0.1725 
			\end{tabular}
			\caption{Minimum and maximum eigenvalues of preconditioned $\bfR_i$ for different choices of $\widehat{\bfR}_i$ computed using the $\texttt{eigs}$ function in {\scshape Matlab}.}
			
			\label{tab:RtildeReigs}
		\end{small}
	\end{table}
	
	Table~\ref{tab:RtildeReigs} shows the extreme eigenvalues of $\widehat{\bfR}_i^{-1}\bfR_i$ for each of the preconditioners discussed in Section \ref{sec:Rprecond}.  We recall that the maximum eigenvalue of $\bfR_{ME}^{-1}\bfR_i$ is $1$ by definition. We fix the smallest eigenvalue of $\bfR_i$, $\lambda_{\min}(\bfR_i)=0.41$, to ensure that $\bfR_i$ is well-conditioned. We see that the maximum eigenvalue of $\bfR_i$ increases with $p$. Eigenvalues are more extreme for $\bfR_{diag}^{-1}\bfR_i$ than for any other preconditioned matrix. Including correlation information is beneficial in terms of the extreme eigenvalues. As the parameter choice for $\bfR_{RR}$ only depends on the smallest eigenvalue of $\bfR_i$ (which is fixed), the minimum and maximum eigenvalues of $\bfR_{RR}^{-1}\bfR_i$ do not change with increasing $p$. The block approach clusters both minimum and maximum eigenvalues of $\bfR_{block}^{-1}\bfR_i$ either side of $1$, whereas the reconditioning approaches { lead to a maximum eigenvalue of the preconditioned matrix} which is very close or equal to $1$. For the approach used here, where $T$ is set to the second smallest eigenvalue of $\bfR_i$, all eigenvalues bar the smallest of $\bfR_{ME}^{-1}\bfR_i$ are equal to $1$. 
	
	{\rev{\subsection{Aspects of numerical linear algebra implementation}\label{sec:NLAimplementation}
			In this section we briefly discuss some of the numerical linear algebra aspects of our implementation to run the numerical experiments in Section~\ref{sec:NumExpt}.
			
			We take advantage of the specific matrix structures in \eqref{eq:saddlepointA}, and store only the non-zero blocks of the matrix, i.e. $\bfR_i\in\mathbf{R}^{p \times p}$, $\bfB,\bfQ_i\in\mathbf{R}^{s \times s}$ and $\bfH_i\in\mathbf{R}^{p \times s}$.  Each of these (relatively) small matrices is stored as a sparse matrix. We also precompute and store $\widehat{\bfR}_i\in\mathbf{R}^{p\times p}$ prior to the iteration of the Krylov subspace method.
			
			We compute the matrix--vector products $\mathbfcal{A}\mathbf{v}$, and the preconditioner solves $\mathbfcal{P}_D^{-1}\mathbf{v}$ and $\mathbfcal{P}_I^{-1}\mathbf{v}$, via the action of a matrix on a vector rather than building the full matrices. We now describe this process briefly. Note first that
			\begin{equation*}
				\mathbfcal{A}\mathbf{v} = \mathbfcal{A}
				\begin{pmatrix}
					\mathbf{v}_1 \\ \mathbf{v}_2\\ \mathbf{v}_3
				\end{pmatrix} := \begin{pmatrix}
					\mathbf{c}_1 \\ \mathbf{c}_2\\ \mathbf{c}_3
				\end{pmatrix} = \begin{pmatrix}
					\bfD\mathbf{v}_1 + \bfL\mathbf{v}_3\\
					\bfR\mathbf{v}_{2} + \bfH\mathbf{v}_3\\
					\bfL^\top\mathbf{v}_1 + \bfH^\top\mathbf{v}_2
				\end{pmatrix}.
			\end{equation*}
			We compute $\mathbf{c}_1,\mathbf{c}_2,\mathbf{c}_3$ separately by looping over the $N+1$ blocks of $\bfD$, $\bfR$, and $\bfH.$ The action of $\bfL$ and $\bfL^{-1}$ on a vector is applied via a function. We note that for each evaluation of the matrix--vector product we require one evaluation each of $\bfL^\top $ and $\bfL$.
			
			To apply the inverse of the block diagonal preconditioner to a vector, we have that
			\begin{equation*}
				\mathbfcal{P}_D^{-1}\mathbf{v} = \begin{pmatrix}
					\widehat{\bfD}^{-1} & \mathbf{0} & \mathbf{0} \\ \mathbf{0} & \widehat{\bfR}^{-1} & \mathbf{0} \\ \mathbf{0} & \mathbf{0} & \widehat{\bfS}^{-1}
				\end{pmatrix} \begin{pmatrix}
					\mathbf{v}_1 \\ \mathbf{v}_2 \\ \mathbf{v}_3 
				\end{pmatrix} = \begin{pmatrix}
					\widehat{\bfD}^{-1}\mathbf{v}_1\\ \widehat{\bfR}^{-1}\mathbf{v}_2 \\ \widehat{\bfL}^{-1}\bfD\widehat{\bfL}^{-\top}\mathbf{v}_3
				\end{pmatrix}.
			\end{equation*}
			We see that each application of $\mathbfcal{P}_D^{-1}$ to a vector requires one evaluation each with $\widehat{\bfD}^{-1},\widehat{\bfR}^{-1},\widehat{\bfL}^{-1}$, and $\widehat{\bfL}^{-\top}$. To apply $\widehat{\bfD}^{-1}$ and $\widehat{\bfR}^{-1}$ we loop over the $N+1$ blocks of the matrices, and for $\widehat{\bfL}^{-1}$ and $\widehat{\bfL}^{-\top}$ we loop over the $k$ blocks of the relevant sub-matrices.  
			Now,
			\begin{equation*}
				\mathbfcal{P}_I^{-1}\mathbf{v} = \begin{pmatrix}
					\mathbf{0} & \mathbf{0} & \widehat{\bfL}^{-\top} \\ \mathbf{0} & \widehat{\bfR}^{-1} & \mathbf{0} \\  \widehat{\bfL}^{-1} & \mathbf{0} & -\widehat{\bfS}^{-1}
				\end{pmatrix}\begin{pmatrix}
					\mathbf{v}_1 \\ \mathbf{v}_2 \\ \mathbf{v}_3
				\end{pmatrix} = \begin{pmatrix}
					\overbrace{\widehat{\bfL}^{-\top}\mathbf{v}_3}^{\mathbf{c}_1}\\
					\widehat{\bfR}^{-1}\mathbf{v}_2\\
					\widehat{\bfL}^{-1}(\mathbf{v}_1-\bfD\mathbf{c}_1)
				\end{pmatrix}.
			\end{equation*}
			By computing the first component of $ \mathbfcal{P}_I^{-1}\mathbf{v}$, $\mathbf{c}_1$, prior to the final component, we can apply the inexact constraint preconditioner with the same dominant costs-per-iteration as an application of the block diagonal preconditioner. This  saves the cost of a second application of $\widehat{\bfL}^{-\top}$. An additional advantage of the inexact constraint preconditioner is that it does not require the application of $\widehat{\bfD}^{-1}$ to a vector, and hence may result in significant computational savings for challenging choices of $\bfD$.

			We note that all of the experiments in Section~\ref{sec:NumExpt} run the Krylov subspace methods to convergence to a { relative} tolerance of $10^{-6}$. This is in contrast to most operational data assimilation implementations, where a small number of fixed iterations are applied. Due to the inclusion of the Lagrange multipliers in the residual of the saddle point formulation, we no longer have monotonic decrease of the objective cost. Running our experiments to convergence ensures a fair comparison between each preconditioner. We leave the investigation of the non-monotonicity for each preconditioner to future work.
			
	}}
	\section{Numerical experiments}\label{sec:NumExpt}
	In this section we present   numerical experiments for the two problems described in Section \ref{sec:NumFramework}. We use the MINRES implementation of \cite{website:MINRES} for the block diagonal preconditioner, with a residual-based convergence criterion in the two norm. 
	{ For the inexact constraint preconditioner we use the GMRES implementation of \cite{greif2011multi} with no restarts, and a convergence criterion given by the relative residual in the two norm.} We use a tolerance of $10^{-6}$ for both problems of interest. 
	
	{{ We note that  all of our experiments converge within the maximum number of iterations we allow ($1000$).}}

	\subsection{Lorenz 96 model}\label{sec:Lorenz96}
	The {\rev{main}} problem of interest concerns the Lorenz 96 model \cite{lorenz1996predictability}, a non-linear problem that is often considered as a test problem for data assimilation applications (see for example \cite{dauvzickaite2020spectral,green2019model} for use within the saddle point formulation for data assimilation). The Lorenz 96 model consists of $s$ coupled ordinary differential equations which are discrete in space and continuous in time. Consider $s$ equally spaced points on the unit line, i.e. $\Delta x = \frac{1}{s}$.  For $i=1,\dots,s$,
	\begin{equation*}
		\frac{{\rm d}x_i}{{\rm d}t} = (x_{i+1}-x_{i-2})x_{i-1}-x_i + 8,
	\end{equation*}
	where we have periodic boundary conditions (so $x_{-1} = x_{s-1}$, $x_{0}=x_{s}$, $x_{s + 1} = x_1$). The choice of forcing constant $F=8$ induces chaotic behaviour, and is typical for data assimilation applications. 
	We use the numerical implementation of \cite{el2015conditioning}, where the model is integrated in time with a fourth-order Runge--Kutta scheme. For all experiments we consider $N=15$ subwindows.  We consider $\Delta t = 10^{-4}$ for all the experiments, 
	although similar results were obtained for other values of $\Delta t$ and are not presented here. 
	
	As we are interested in assessing the performance of preconditioners within the linearised inner loops, we consider a single outer loop of the weak constraint formulation.  {{The Lorenz 96  example can  be considered as a study of how well the proposed preconditioners perform in a realistic setting where the linearised model operators $\bfM_i$  differ in each subwindow.}} { We note that the setting of of Proposition \ref{prop:evalsLML} holds}, and hence using $\LM$ guarantees that $\bfL_M^{-\top}\bfL^\top\bfL\bfL_M^{-1}$ possesses a number of unit eigenvalues depending on the choice of $k$. However, we cannot bound the maximum eigenvalue of $\bfL_M^{-\top}\bfL^\top\bfL\bfL_M^{-1}$ using the theory of Section \ref{sec:ApproxLhat}: the assumptions of Proposition \ref{prop:lowrankallM} are not satisfied as  $\lambda_{\max}(\bfM_{15}^\top\bfM_{15})>1$ for all choices of $\Delta x$ that were studied.

	\begin{table}[]
		\begin{small}
			\centering
			\begin{tabular}{c c|cccc}
				$\widehat{\bfL}$ &$\widehat{\bfR}$ & $\lambda_{\min}(\mathbfcal{P}_D^{-1}\mathbfcal{A})$& $\max(\lambda(\mathbfcal{P}_D^{-1}\mathbfcal{A})<0)$ &$\min(\lambda(\mathbfcal{P}_D^{-1}\mathbfcal{A})>0)$ &$\lambda_{\max}(\mathbfcal{P}_D^{-1}\mathbfcal{A})$ \\ \hline 
				$\bfL_M$ & $\bfR_{block}$&  $-15.3994$ &  $-0.1233$&    0.7886 &  16.4061\\
				$\bfL_M$ & $\bfR_{RR}$&   $-14.2002$  & $-0.0574$ &   0.2908  & 14.8536\\
				$\bfL_M$ & $\bfR$&  $-13.8536$  & $-0.1773$  &  1.0000 &  14.8536\\
				$\bfL_0$ & $\bfR$&  $-12.2360$  & $-0.0551$   & 1.0000&   13.2360\\
				\hline
				$\bfL_M$ & $\bfR_{block}$& $-3.5516$ &  $-0.1773$ &   0.8295 &   4.5535\\
				$\bfL_M$ & $\bfR_{RR}$& $-2.5351$  & $-0.1773$  &  0.3878  &  3.2438\\
				$\bfL_M$ & $\bfR$&$-3.6119$  & $-0.1773$ &   1.0000 &   4.6119\\
				$\bfL_0$ & $\bfR$&$-2.0415$ &  $-0.0551$ &   1.0000 &   3.0415\\
			\end{tabular}
			\caption{{\rev{Experiment A:}} (Top:) Bounds on negative and positive eigenvalues from Theorem \ref{theorem:BDPrec} 
				for $s=500$, $N=5$, $k=3$ {\rev{with parameters from Experiment A in Table \ref{tab:ExperimentalDesign}}}; 
				(Bottom:) Extreme negative and positive eigenvalues for this problem computed using the $\texttt{eigs}$ function in {\scshape Matlab}.}
			
			\label{tab:BoundsondiagI}
		\end{small}
	\end{table}
	
	Table~\ref{tab:BoundsondiagI} shows the values of the bounds from Theorem~\ref{theorem:BDPrec} {(top four rows)} and the computed eigenvalues {(bottom four rows)} when using $\mathbfcal{P}_D$ with $\bfL_M(3)$, and {a number of} choices for $\widehat{\bfR}$ introduced in Section~\ref{sec:Rprecond}.  We note that these experiments consider $\bfD =\mathbf{I}$, as the condition number of $\mathbf{D}$ appears in all terms in the bound. In the more realistic case where $\bfB,\bfQ_i \neq \mathbf{I}$ we expect the bounds given by Theorem~\ref{theorem:BDPrec} to be much weaker. We note that even with this choice of $\bfD$ the bounds provide pessimistic estimates of the eigenvalues.
	{ However, the qualitative behaviour of the bounds is similar to that of the computed eigenvalues for different choices of $\widehat{\mathbf{R}}$ and $\mathbf{L}_M$. The eigenvalues of smallest absolute value  are mainly affected by the choice of $\mathbf{L}_M$ (for the negative eigenvalue with largest magnitude) or $\widehat{\mathbf{R}}$ (for the smallest positive eigenvalue). This is not true for the bounds in the case of the negative eigenvalue with smallest magnitude. The computed eigenvalues of largest magnitude and the corresponding bounds are affected by changes to both $\mathbf{L}_M$ and $\widehat{\mathbf{R}}$, but these changes are rather small in all cases.} As the largest magnitude computed eigenvalues are small for this problem (all less than five), the improvements to the small magnitude eigenvalues with $\bfL_M$ and correlated choices of $\widehat{\bfR}$ are likely to have {the most significant} influence on the convergence of the iterative methods.

	\begin{figure}
		\centering
		\includegraphics[width=0.95\textwidth,clip,trim=0mm 0mm 20mm 3mm]{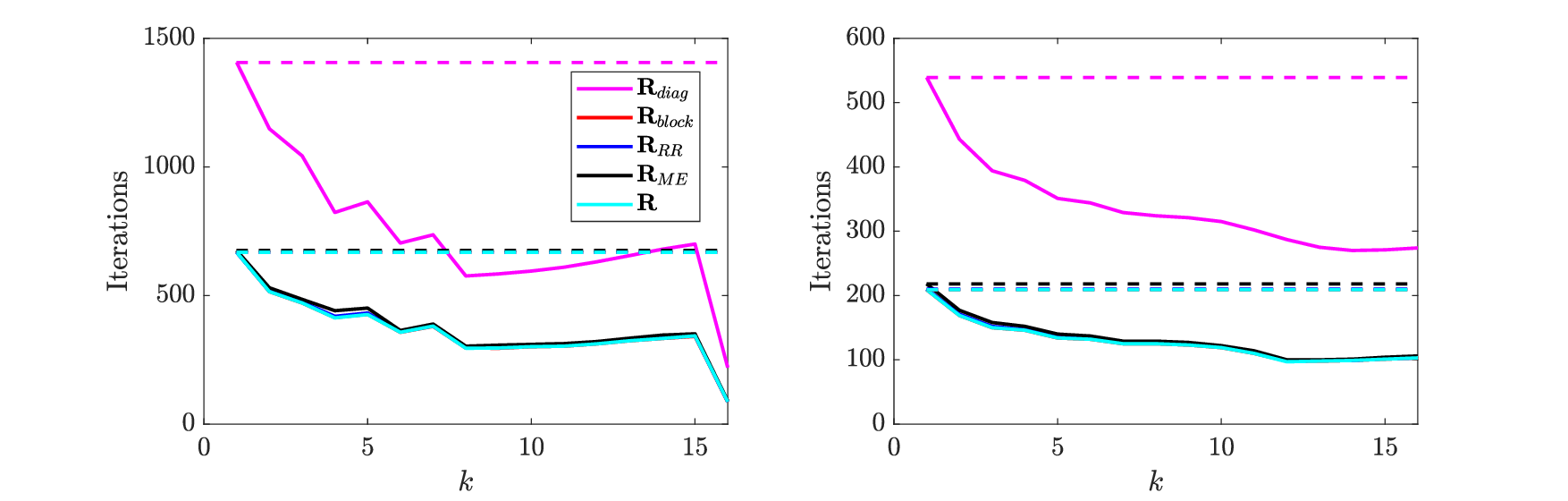}
		\caption{{{Experiment A}}: Number of iterations required for convergence of MINRES  for different choices of $\widehat{\bfR}$ within the block diagonal preconditioner $\mathbfcal{P}_D$ ({left}) and the  the inexact constraint preconditioner $\mathbfcal{P}_I$ ({right}) for the Lorenz 96 problem, with increasing $k$. The dimension of the problem is given by $\mathbfcal{A}\in\mathbb{R}^{100,000\times 100,000}$. The dashed line shows the results for $k=1$ i.e. $\bfL_M = \bfI = \bfL_0$. We note that $k=16$ corresponds to $\bfL_M \equiv \bfL$. {{Legend entries correspond to different choices of $\widehat{\bfR}$} as described in Section \ref{sec:Rprecond}}.}
		\label{fig:LorenzDiag}
	\end{figure}
	
	\begin{table}[]
		\begin{small}
			\centering
			\begin{tabular}{c|cccc|ccccc}
				$k$ &$\bfR_i$    &   $\bfD_i$    &  $\widehat{\bfD}_i^{-1}$  &   $\bfM_i/\bfM_i^\top$ &$\bfR_i$  &$\bfR_{block}^{-1}$  &    $\bfD_i$    &  $\widehat{\bfD}_i^{-1}$  &   $\bfM_i/\bfM_i^\top$\\
				\hline
				1&    22496 &   44992&    22496&    42180&    10704&    10704&    21408&    10704 &   20070\\ 
				2&   18368&   36736 &   18368  &  34440 &    8240&     8240 &   16480  &   8240  &  15450\\  
				3&    16688&   33376 &   16688&    47978  &   7536    & 7536 &   15072    & 7536 &   21666\\
				4&    13168  &  26336  &  13168  &  41150&     6624  &   6624  &  13248 &    6624&    20700\\
				7&   11776  &  23552  &  11776  &  41216    & 6080  &  6080 &   12160 &    6080&    21280\\
				10&     9520  &  19040  &   9520  &  34510    & 4816  &   4816   &  9632 &    4816&    17458\\
				16&     3520  &  7040 &   3520 &  12760    &1376  &  1376 &   2752 &1376 &   4988
			\end{tabular}
			\caption{Experiment A: Total number of matrix--vector products with component matrices for $\mathbfcal{P}_{D}$ for increasing $k$ for $\bfR_{diag}$ (left) and $\bfR_{block}$ (right).}
			\label{tab:Lorenz96PD}
		\end{small}
	\end{table}
	
	We now study the performance of the proposed choices of preconditioners for the Lorenz 96 problem {{(Experiment A in Table~\ref{tab:ExperimentalDesign})}}. 
	{The left panel of} Figure~\ref{fig:LorenzDiag} shows how the  number of iterations  required for convergence changes with the choice of $\widehat{\bfL}$ and $\widehat{\bfR}$ within $\mathbfcal{P}_D$ as $k$ increases. Including some model information in $\widehat{\bfL}$ leads to a reduction in iterations  compared to $k=1$ for all choices of $\widehat{\bfR}$. For  $k<5$ including more model information leads to faster convergence. However, for $k\ge5$ the change in iterations is non monotonic. Including correlation information in $\widehat{\bfR}$ results in large improvements to convergence.  { We note that the lines corresponding to the correlated choices of $\widehat{\bfR}$ ($\bfR_{block}$, $\bfR_{RR}$, and $\bfR_{ME}$) lie almost directly on top of the line for $\bfR$.} Indeed, if using $\bfR_{diag}$ we require $k\ge 8$ to obtain fewer iterations than using $\bfL_0$ with an improved choice of $\widehat{\bfR}$. There is very little difference {in performance} between the correlated choices of $\widehat{\bfR}$. 
	For all choices of $\widehat{\bfR}$ the smallest number of iterations occurs when using $k=N+1$, i.e. the exact choice of $\bfL$.  {{Table~\ref{tab:Lorenz96PD} shows the number of matrix--vector products required to reach convergence for the block diagonal preconditioner {\rev{using $\bfR_{diag}$ and $\bfR_{block}$. Results using $\bfR_{RR}$ and $\bfR_{ME}$ are  similar to those with $\bfR_{block}$}}. The number of matrix--vector products with $\bfD_i$, $\widehat{\bfD}_i^{-1}$ and $\bfM_i$ is reduced when using $\bfR_{block}$ compared to $\bfR_{diag}$. For some choices of $k$ the {total} number of evaluations with $\bfR_i$ and $\bfR_{block}^{-1}$ is slightly larger than in the $\bfR_{diag}$ case. 
			Increasing $k$ broadly decreases the number of matrix--vector products with the error covariance matrices and their inverses. For some choices of $k>1$ more evaluations of $\bfM_i$ and $\bfM_i^\top$ are required than when using $\bfL_0$. However, this increase is small compared to the decrease in the other components. }}
	
	\begin{table}[]
		\begin{small}
			\centering
			\begin{tabular}{c|ccc|cccc}
				$k$ &$\bfR_i$    &   $\bfD_i$     &   $\bfM_i/\bfM_i^\top$ &$\bfR_i$  &$\bfR_{block}^{-1}$  &    $\bfD_i$      &   $\bfM_i/\bfM_i^\top$\\
				\hline
				1&        8624 &   17248 &   16170&    3344 &   3344 &    6688 &   6270\\
				2&    7088 &   14176 &   13290&    2704 &   2704 &    5408 &   5070\\
				3&   6304 &   12608 &   18124&    2400 &   2400 &    4800 &   6900\\
				4&  6064 &   12128 &   18950&    2336 &   2336 &    4672 &   7300\\
				7&    5264 &   10528 &   18424&    2000&    2000 &    4000&    7000\\
				10&    5040 &   10080 &   18270&    1904&    1904 &    3808&    6902\\
				16&   4384 &    8768 &   15892&    1648&    1648 &    3296&    5974\\
			\end{tabular}
			\caption{Experiment A: Total number of matrix--vector products with component matrices for $\mathbfcal{P}_{I}$ for increasing $k$ for $\bfR_{diag}$ (left) and $\bfR_{block}$ (right).}
			\label{tab:Lorenz96IC}
		\end{small}
	\end{table}
	
	{The right panel of} Figure \ref{fig:LorenzDiag} shows how the  number of iterations required for convergence changes with the choice of $\widehat{\bfL}$ and $\widehat{\bfR}$ within $\mathbfcal{P}_I$ as $k$ increases. We see a clear benefit of including model information in terms of a reduction in iterations. Increasing $k$ leads to a reduction in the number of iterations required for convergence when using $\bfL_M$ for all choices of $\widehat{\bfR}$, unlike when using $\mathbfcal{P}_D$.    The benefit of using an improved estimate of $\widehat{\bfR}$ is even more stark, with any choice of correlated $\widehat{\bfR}$ and $\bfL_0$ leading to fewer iterations than $\bfR_{diag}$ even when using $\widehat{\bfL} \equiv \bfL$. {Again, the lines for the correlated choices of $\widehat{\bfR}$ ($\bfR_{block}$, $\bfR_{RR}$, and $\bfR_{ME}$) lie almost directly on top of the line for $\bfR$.}
	{{Table~\ref{tab:Lorenz96IC} shows the total number of matrix--vector products required to reach convergence for the inexact constraint preconditioner {\rev{for $\bfR_{diag}$ and $\bfR_{block}$. Results using $\bfR_{RR}$ and $\bfR_{ME}$ are  similar to those with $\bfR_{block}$}}. In this case, using $\bfR_{block}$ leads to a large reduction in the number of matrix--vector products for all components. Increasing $k$ can lead to increases in the number of model matrix--vector evaluations, but leads to decreases in the total number of matrix--vector products. }}

	Table~\ref{tab:Lorenz96BigExamplePD}  shows the performance of the block diagonal preconditioner and inexact constraint preconditioner, respectively, for a higher-dimensional problem when using $\bfR_{block}$, $\bfR_{RR}$, and $\bfR$ itself to approximate $\bfR$ within the preconditioner.  
	Similarly to the smaller dimensional problem considered in Figure \ref{fig:LorenzDiag}  using $\bfL_M$  leads to improved convergence in terms of iterations compared to $\bfL_0$. The different choices of $\widehat{\bfR}$ lead to comparable iteration numbers and we recall that $\bfR_{block}$ has additional {sparsity} structure.  Increasing $k$ leads to a slight reduction in the number of iterations, but increases the computational cost of each iteration. For this problem, choosing $k=3$ or $4$ allows decreased iteration counts compared to $\bfL_0$, { without too many more matrix--vector products with $\mathbf{M}_i$.} Iteration counts are much smaller for the inexact constraint preconditioner than the block diagonal preconditioner. Overall, using our new preconditioners ${\bfL_M}$ and correlated choices of $\widehat{\bfR}$ result in fewer iterations and {matrix--vector products} compared to those obtained when using $\bfL_0$ or $\bfR_{diag}$ for the Lorenz 96 problem. 
	
	\begin{table}[]
		\begin{small}
			\centering
			\begin{tabular}{c |ccc|ccc}
				& $\bfR_{block}$ & $\bfR_{RR}$ & $\bfR$& $\bfR_{block}$ & $\bfR_{RR}$ & $\bfR$\\ \hline
				$\bfL_0$  & 759 & 822 & 822 & 359 & 275 & 275\\
				$\bfL_M$, $k=3$ & 433 & 466& 467& 244 &205&205 \\
				$\bfL_M$, $k=4$ & 348 & 335 & 336& 228 & 200 &200\\
				$\bfL_M$, $k=5$ & 367 & 354 & 355& 206 & 182 & 182\\
			\end{tabular}
			\caption{{{Experiment A}}: Number of iterations required for convergence of {{MINRES}} with the block diagonal preconditioner $\mathbfcal{P}_D$ (left) and $\mathbfcal{P}_I$ (right) applied to the Lorenz 96 problem, using $\bfR_{block}$, $\bfR_{RR}$, $\bfR$ in combination with $\bfL_0$, $\bfL_M$ ($k= 3,4,5$). Here,  
				$\mathbfcal{A}\in \mathbb{R}^{1,600,000\times 1,600,000}$.}
			
			\label{tab:Lorenz96BigExamplePD}
		\end{small}
	\end{table}

	\subsection{Heat equation with Dirichlet boundary conditions}
	\label{sec:HeatEquation}
	The second problem of interest that we consider here is the one-dimensional heat equation on the unit line 
	\begin{equation}\label{eq:heatequation}
		\frac{\partial u}{\partial t} = \alpha \, \frac{\partial^2 u}{\partial x^2},
	\end{equation}
	with { homogeneous Dirichlet boundary conditions}. We discretise \eqref{eq:heatequation} using the forward Euler method in time and second-order centred differences in space. This means we can write the model evolution in matrix form for a single time model step as  $\mathbf{u}^{t+\Delta t} = \bfM_{\Delta t} \mathbf{u}^{t}$,
	where $\bfM_{\Delta t}$ denotes the application of a single model time-step of length $\Delta t$ to the heat equation with Dirichlet boundary conditions; this is given by 
	\begin{equation*}\label{eq:MheateqnzeroBC}
		\bfM_{\Delta t} = \begin{pmatrix}
			0& 0 & 0 & 0 &\cdots &  0 \\
			0 & 1-2r & r & 0 &  \cdots& 0 \\
			0 & r & 1-2r & \ddots & & \vdots \\
			0 & 0 & \ddots & \ddots & r & 0 \\
			\vdots & \vdots && r & 1-2r & 0  \\
			0 & 0 & \cdots  &0& 0& 0 
		\end{pmatrix},
	\end{equation*}
	where $r=\frac{\alpha \Delta t}{(\Delta x)^2}$. {The case of non-homogeneous boundary conditions would follow similarly, by applying a source term to the model evolution equation.} For the numerical experiments presented here we fix $\alpha=1$ and  vary spatial and temporal resolutions together, setting the ratio $r=\frac{\Delta t}{(\Delta x)^2} = 0.4$ for all experiments.
	\begin{figure}
		\centering
		\includegraphics[width = 0.95\textwidth, clip,trim=10mm 0mm 20mm 0mm]{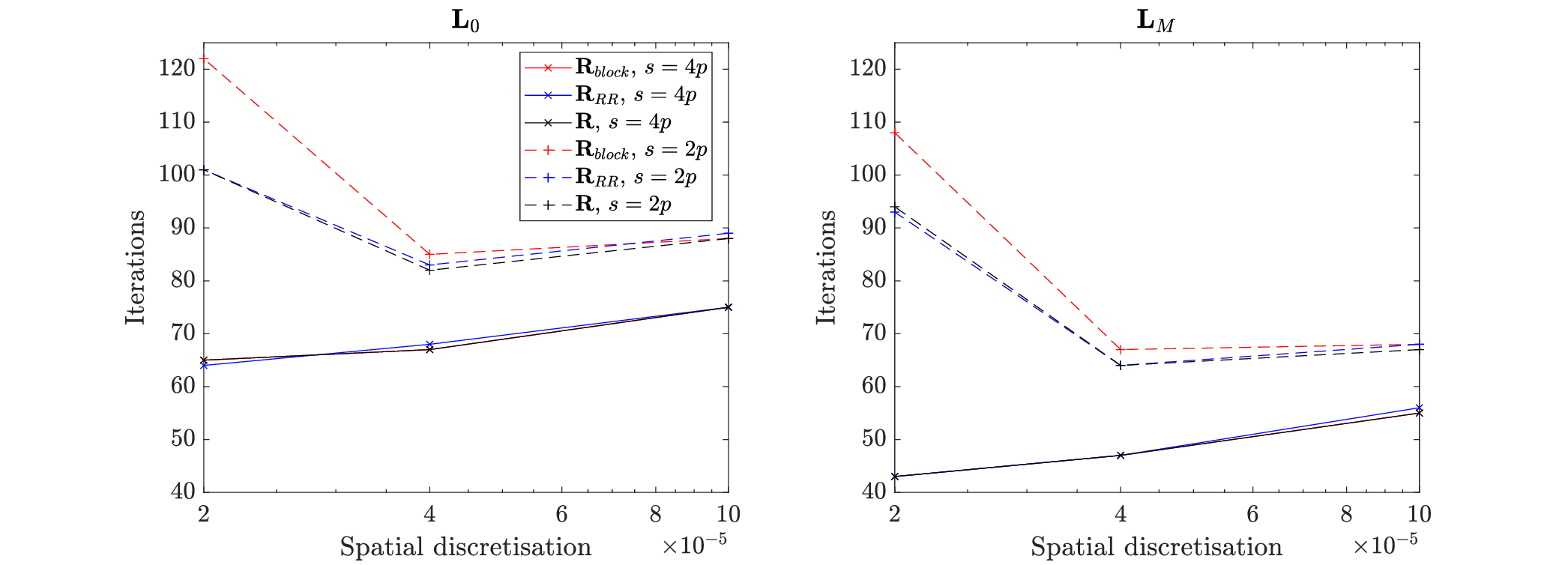}
		\caption{{\rev{Experiment B}}: Number of iterations required for convergence using $\bfL_0$ ({left}) and $\bfL_M$ ({right}) for different choices of $\widehat{\bfR}$ within the inexact constraint preconditioner $\mathbfcal{P}_I$ for the heat equation, for large choices of $s$. Two observation ratios are shown: $p = \frac{s}{2}$ (dashed line) and $p=\frac{s}{4}$ (solid line). The dimension of the problem ranges from $\mathbfcal{A} \in \mathbb{R}^{150,000\times150,000}$ (right) to $\mathbfcal{A} \in \mathbb{R}^{750,000\times750,000}$ (left) for $p = \frac{s}{2}$, and $\mathbfcal{A} \in \mathbb{R}^{135,000\times 135,000}$ (right) to $\mathbfcal{A} \in \mathbb{R}^{675,000\times 675,000}$ (left) for $p = \frac{s}{4}$. {{Legend entries correspond to different choices of $\widehat{\bfR}$} as described in Section 5.}}
		\label{fig:ConvergenceICLarge}
	\end{figure}
	
	{\rev{One advantage of the heat equation test problem is the ability to consider how our new preconditioners scale with problem size. We now consider the best-performing preconditioners for a high-dimensional example, namely the inexact constraint preconditioner for $\bfR_{block},\bfR_{RR}$, and $\bfR$. }} 
	{{Experiment B studies}} the behaviour of the inexact constraint preconditioner for high-dimensional problems. In particular, for $\Delta x=2\times10^{-5}$ the dimension of the full saddle point problem is $750,000\times 750,000$ for $s=2p$. 
	In Figure \ref{fig:ConvergenceICLarge} we only consider the inexact constraint preconditioner and the three best choices of $\widehat{\bfR}$. We find that even for a high-dimensional problem the number of iterations is small, {with only a modest difference between the results for $\bfR_{block}$ and $\bfR$ (as well as $\bfR_{RR}$)}. Similarly to the lower dimensional case, using $\bfL_M$ requires fewer iterations than using $\bfL_0$. 
	Figure  \ref{fig:ConvergenceICLarge} also considers two observation networks. We see that using a larger number of observations requires a larger number of iterations to reach convergence, which coincides with the findings of \cite{dauvzickaite2020spectral} for the unpreconditioned case. However, the qualitative behaviour across the different choices of $\widehat{\bfL}$ and $\widehat{\bfR}$ is the same for both observation networks. 
	
	\section{Conclusions}\label{sec:Conclusions}
	
	{ We proposed new preconditioners for the saddle point formulation of the weak-constraint 4D-Var data assimilation problem {{in the presence of correlated observation errors}}. Our approach for approximating the model term, $\widehat{\bfL}$, incorporated model information for the first time. We also proposed a range of approaches that permit inclusion of computationally efficient correlation information within the observation error covariance term, $\widehat{\bfR}$.  In summary:
		\begin{itemize}
			\item We developed new bounds for the eigenvalues of the preconditioned saddle point system in the case of a block diagonal preconditioner. 
			\item We investigated how the constituent terms within the bounds behave for existing and proposed choices of $\widehat{\bfL}$ and $\widehat{\bfR}$. 
			Including model information via $\LM$ yields many repeated unit eigenvalues of $\LM^{-\top}\bfL^\top\bfL\LM^{-1}$.  Our new approaches yield eigenvalues of this matrix that are frequently bounded above by moderate numbers. 
			\item We considered two numerical examples: the Lorenz 96 problem and the heat equation. Including model information via $\LM$ reduced iterations for both problems. 
		\end{itemize}
		
		The use of preconditioners that account for correlated observation information led to a significant reduction in iterations for all experiments. 
		For many problems where $\bfR$ is very ill-conditioned, we would expect the improvements in performance to be even greater than in the experiments presented here. 
		We find that including additional model information in $\bfL_M$ leads to reduced iterations, but increases the computational expense of each iteration. We therefore suggest that selecting $k=3$ or $k=4$ represents a sensible trade-off. 
		Future work for this problem includes developing efficient approximations of $\bfD$, multi-core implementations of our new preconditioners, and experiments within a full-scale operational NWP system.}

	\section*{Acknowledgements}
	We thank Adam El-Said for his code for the Lorenz 96 weak constraint 4D-Var assimilation problem. We gratefully acknowledge funding from the Engineering and Physical Sciences Research Council (EPSRC) grant EP/S027785/1.
	
	\section*{{Appendix: Block preconditioner for $\widehat{\mathbf R}$}}
	
	{ We state below the algorithm used to apply the block preconditioner $\bfR_{block}$ for $\bfR$, as described in Section \ref{sec:parallelR}.}
	
	\begin{algorithm}[H]
		\caption{Block preconditioner for $\widehat{\mathbf R}$\label{alg:Block}}
		Inputs: $\bfR_i$, $\texttt{pvec} =$ vector of block sizes, $\texttt{tol}=$ tolerance for retaining blocks, $\texttt{maxsize} =$ maximum size permitted on a single processor, $\texttt{numproc} =$ number of available processors.\\
		Compute $\texttt{p} = \texttt{sum(pvec)}$, $\texttt{pn} = |\texttt{pvec}|$. \\
		Define $\texttt{pst} =$ starting index for each new block. \\
		Initialise $\bfR_{block} = \bfR_i$.\\
		\textbf{for} {$j =$ \texttt{1:pn-1}}\\
		~~~~Compute scaled Frobenius norm of super-diagonal blocks via \\
		
		~~~~\texttt{normvec(j) = 1/sqrt(pvec(j)*pvec(j+1))*} \\
		~~~~~~~~\texttt{norm(R(pst(j):pst(j+1)-1,pst(j+1):pst(j+2)-1),`fro')}.\\
		\textbf{end} \\
		Retain blocks where \texttt{normvec(j) >= tol}:\\
		\textbf{for} {$j=$ \texttt{1:pn-1}}\\
		~~~~\textbf{if} { \texttt{normvec(j) < tol}}\\
		
		~~~~~~~~{Set $\bfR_{block}\texttt{(pst(j+1):p,1:pst(j+1)-1) = 0}$.}\\
		~~~~~~~~{Set $\bfR_{block}\texttt{(1:pst(j+1)-1,pst(j+1):p) = 0}$.}\\
		~~~~\textbf{end}\\
		\textbf{end}\\
		\textbf{if} {size of largest block $> \texttt{maxsize}$}\\
		~~~~{Split largest block into two components.}\\
		\textbf{end}\\
		\textbf{if} {number of distinct blocks $> \texttt{numproc}$} \\
		~~~~{Combine two smallest adjacent blocks in $\bfR_{block}$.}\\
		\textbf{elseif} {number of distinct blocks $< \texttt{numproc} - 2$}\\
		~~~~{Split largest block of $\bfR_{block}$ into two components.}\\
		\textbf{end}\\
	\end{algorithm}

	\bibliographystyle{plain}
	\bibliography{bibPrecond}
\end{document}